\documentclass[11pt,a4paper]{article}
\usepackage[T1]{fontenc}
\usepackage[utf8]{inputenc}
\usepackage{lmodern}
\usepackage[margin=28mm]{geometry}
\usepackage{amsmath,amssymb,amsthm,mathtools,mathrsfs}
\usepackage{microtype}
\usepackage{titling}
\usepackage{booktabs}
\usepackage{flafter}
\usepackage{needspace}
\usepackage{tikz}
\usetikzlibrary{arrows.meta,positioning}
\usepackage{xurl}
\usepackage[colorlinks=true,allcolors=blue!45!black]{hyperref}
\usepackage{bookmark}
\numberwithin{equation}{section}
\newtheorem{theorem}{Theorem}[section]
\newtheorem{proposition}[theorem]{Proposition}
\newtheorem{lemma}[theorem]{Lemma}
\newtheorem{corollary}[theorem]{Corollary}
\theoremstyle{remark}

\DeclareMathOperator{\tr}{tr}
\DeclareMathOperator{\Sym}{Sym}

\newcommand{\RP}{\mathbb{RP}}
\newcommand{\R}{\mathbb R}
\newcommand{\cC}{\mathcal C}

\newcommand{\one}{\mathbf 1}
\newcommand{\ip}[2]{\langle #1,#2\rangle}
\allowdisplaybreaks[2]
\pretitle{\centering\LARGE\bfseries}
\posttitle{\par\vspace{1em}}
\preauthor{\centering\normalsize\begin{tabular}{@{}c@{}}}
\postauthor{\end{tabular}\par\vspace{0.6em}}
\predate{\centering\small}
\postdate{\par}

\title{An orthogonality inequality in \(\RP^{n-1}\)\\
and the double-cap conjecture}
\author{%
Allen Ibiapina\textsuperscript{1}\thanks{Corresponding author:
\href{mailto:allenibiapina@ufc.br}{\nolinkurl{allenibiapina@ufc.br}}.}
\enspace
Carlos André Gomes\textsuperscript{2}
\enspace
Alan Pio\textsuperscript{1}
\enspace
Samuel Belo\textsuperscript{1}
\enspace
Ernandes Ferreira\textsuperscript{1}\\[0.7em]
\footnotesize\textsuperscript{1}\,Departamento de Matemática,
Universidade Federal do Ceará, Fortaleza, Brazil\\[0.25em]
\footnotesize\textsuperscript{2}\,Departamento de Matemática,
Universidade de São Paulo, São Paulo, Brazil
}

\date{19 October 2026}

\hypersetup{
 pdftitle={An orthogonality inequality in RP\string^(n-1)
 and the double-cap conjecture},
 pdfauthor={Allen Ibiapina, Carlos André Gomes, Alan Pio,
 Samuel Belo, Ernandes Ferreira},
 pdfsubject={Sharp orthogonal neighborhood inequality
 in real projective space and the double-cap conjecture}
}
\hypersetup{
 pdftitle={An orthogonality inequality in RP\string^(n-1) and the double-cap conjecture},
 pdfauthor={Allen Ibiapina, Carlos Gomes, Alan Pio, Samuel Belo, Ernandes Ferreira},
 pdfsubject={Sharp orthogonal neighborhood inequality in real projective space and the double-cap conjecture}}
\begin{document}

\maketitle

\begin{abstract}
The orthogonal neighborhood of a set in real projective space is the union
of the hyperplanes perpendicular to its points. We prove that caps
minimize neighborhood measure among measurable sets of fixed measure in \(\RP^{n-1}\), for \(n\ge3\), with rigidity up to null sets.
This yields the double-cap conjecture: a measurable subset of
\(S^{n-1}\) containing no orthogonal pair has measure at most that of
two opposite open caps of radius \(\pi/4\). The argument combines
compactness of the Funk transform, rotation comparisons, and projective
variations. We also obtain qualitative stability and determine the
measurable fractional chromatic number of the spherical orthogonality
graph. AI played an assistive role in preparing this paper. The formalization
and verification workflow for a separate Lean 4 proof of the main inequality
and double-cap result, based on a different argument, were carried out
almost entirely using AI-assisted tools.
\end{abstract}

\section{Introduction}

The spherical isoperimetric inequality says that caps have the smallest
perimeter among sets of prescribed volume. Symmetrization and
stability of this comparison are central themes in the geometry of the
sphere; see Sperner~\cite{Sperner} and B\"ogelein, Duzaar, and
Fusco~\cite{BogeDuzaarFusco}. We study a related extremal problem
in which perimeter is replaced by the measure of the
\emph{orthogonal neighborhood}: the set of all directions perpendicular
to at least one point of the original set. This quantity measures
how many directions the set excludes if orthogonality is forbidden.

For spherical convex bodies contained in an open hemisphere, Gao,
Hug, and Schneider~\cite{GaoHugSchneider} proved that caps minimize
neighborhood measure at fixed volume. Their result belongs to the theory
of spherical mean width and polar-volume inequalities. This comparison
naturally raises the question of extending the same bound to arbitrary
measurable sets in real projective space. We prove this extension and
characterize the equality cases.

The principal application resolves Witsenhausen's problem~\cite{Witsenhausen}
and the double-cap conjecture~\cite{Kalai}: a measurable subset of a
sphere containing no orthogonal pair has measure at most that of two
opposite open caps of angular radius \(\pi/4\). A set containing no
orthogonal pair is disjoint from its neighborhood, so a lower bound for
neighborhood measure gives an upper bound for the measure of the set.
The full inequality also controls the measures of two different sets
when no point of one is orthogonal to a point of the other.

To state the result, let \(n\ge3\), let \(S^{n-1}\subset\R^n\) carry
surface measure \(\sigma\) normalized to one, and write
\(M=\RP^{n-1}\). A point of \(M\) is a line through the origin;
the quotient map \(\pi:S^{n-1}\to M\) identifies \(x\) with \(-x\).
We write \([x]=\pi(x)\) and give \(M\) its normalized Riemannian
volume \(\mu\), so that \(\mu(A)=\sigma(\pi^{-1}A)\).
The \emph{orthogonal neighborhood} of \(A\subset M\) is
\begin{equation}\label{eq:shadow}
 N(A)=\{y\in M:x\cdot y=0\text{ for some }x\in A\}.
\end{equation}
Projective points in this formula are represented by unit vectors.
Changing the sign of either representative leaves orthogonality
unchanged, so \(N(A)\) is well defined and is the union of the
projective hyperplanes perpendicular to the points of \(A\).
This is the graph neighborhood for adjacency defined by orthogonality.

A projective cap of radius \(r\) and unit axis \(z\) is
\(\cC_r(z)=\{x:|x\cdot z|\ge\cos r\}\). Its spherical lift consists
of two opposite caps, and its neighborhood is the belt
\(\{y:|y\cdot z|\le\sin r\}\). With
\(c_n=2\Gamma(n/2)/(\sqrt\pi\,\Gamma((n-1)/2))\), their measures
and the resulting neighborhood profile are
\begin{align}
 A_n(r)&=c_n\int_0^r\sin^{n-2}\theta\,d\theta,&
 B_n(r)&=c_n\int_0^r\cos^{n-2}\theta\,d\theta,
 \quad 0\le r\le\pi/2,\label{eq:angular}\\
 g_n(p)&=B_n(A_n^{-1}(p)),\quad 0\le p\le1.
 \label{eq:g}
\end{align}
Thus \(g_n(p)\) is the neighborhood measure of a cap of measure \(p\).
The functions \(A_n\) and \(B_n\) are increasing homeomorphisms
onto \([0,1]\), related by
\begin{equation}\label{eq:complement}
 B_n(r)=1-A_n(\pi/2-r).
\end{equation}
All measures are completed, so subsets of null sets are measurable.
We write
\(\mu_*(D)=\sup\{\mu(K):K\subset D,\ K\text{ compact}\}\)
for inner measure; this equals \(\mu(D)\) when \(D\) is measurable.
Using inner measure avoids a measurability assumption on the
pointwise neighborhood of a completed-measurable set.

\begin{theorem}\label{thm:shadow}
Let \(n\ge3\). Every completed-measurable \(A\subset\RP^{n-1}\)
satisfies
\begin{equation}\label{eq:main}
 \mu_*(N(A))\ge g_n(\mu(A)).
\end{equation}
Projective caps attain the bound. If \(0<\mu(A)<1\) and equality
holds, then \(A\) agrees almost everywhere with a projective cap.
\end{theorem}

The geometric meaning of the convex comparison is particularly
transparent in dimension two. On \(\RP^2\), where
\(g_3(p)=\sqrt{p(2-p)}\), polarity identifies neighborhood measure
with perimeter for convex sets, and spherical isoperimetry gives
the comparison.
Let \(K\subset S^2\) be a spherical convex body with nonempty interior,
contained in an open hemisphere. Write \(|K|\) for its unnormalized
area and \(L(K)\) for its perimeter. Its positive polar \(K^\circ\)
consists of the points \(y\in S^2\) satisfying \(x\cdot y\ge0\)
for every \(x\in K\). Spherical isoperimetry and the area--perimeter
identities for polar bodies give
\begin{align}
 L(K)^2&\ge |K|(4\pi-|K|),\label{eq:spherical-isoperimetry}\\
 |K|+L(K^\circ)&=2\pi,\qquad
 L(K)+|K^\circ|=2\pi.\label{eq:polar-identity}
\end{align}
The first inequality is sharp precisely for caps \cite{BogeDuzaarFusco}.
The polar identities follow from Gauss--Bonnet, first for smooth strictly
convex bodies and then by approximation; see \cite{NagyVigh}.

For \(A=\pi(K)\), projection preserves area, so \(|A|=|K|\).
A great circle misses the connected set \(K\) exactly when its defining
inner product has a constant strict sign on \(K\). The complement of
\(N(A)\) is therefore \(\pi(\operatorname{int}K^\circ)\). Since
\(|\RP^2|=2\pi\), \eqref{eq:polar-identity} gives
\(|N(A)|=2\pi-|K^\circ|=L(K)\), and substitution into
\eqref{eq:spherical-isoperimetry} yields
\begin{equation}\label{eq:planar-shadow}
 |N(A)|^2\ge |A|(4\pi-|A|).
\end{equation}
Dividing by \((2\pi)^2\) recovers the profile \(g_3\).
Neighborhood measure thus takes the place of perimeter in an inequality
that can be stated without convexity or boundary regularity.

The convex comparison also connects with Urysohn's inequality. The mean width of a Euclidean convex body is the average,
over directions, of the distance between its parallel supporting
hyperplanes; at fixed volume, balls minimize this average. Together
with inequalities for polar bodies, this is part of the
Brunn--Minkowski theory; see Schneider~\cite{Schneider}. On the
sphere, the probability that a uniformly random great hypersphere
meets a convex body plays the corresponding role. For
\(K\subset S^{n-1}\) contained in an open hemisphere, this probability
is \(\mu(N(\pi(K)))\), or \(2U_1(K)\) in the normalization of
Gao, Hug, and Schneider~\cite{GaoHugSchneider}. Their inequality
\(U_1(K)\ge U_1(C)\), for a cap \(C\) of equal volume, is therefore
a spherical Urysohn inequality. Through polarity it also says that caps
maximize polar volume at fixed volume, a spherical Blaschke--Santal\'o
inequality. Theorem~\ref{thm:shadow} extends the hitting comparison
beyond the class of convex bodies.

The interpretation by random hyperspheres also places the inequality
in stochastic geometry. Hack and Pivovarov~\cite{HackPivovarov}
proved that, among compact sampling sets of equal positive volume
in an open hemisphere, caps minimize the expected hitting functional
of the convex hull of a fixed number of independent uniform points.
Hug and Reichenbacher~\cite{HugReichenbacher} obtained quantitative
stability estimates for the convex hitting inequality and applied them
to cells of random spherical tessellations. Such stability results
ask whether a set whose functional is nearly minimal must be close
to a cap. This question is also central to quantitative spherical
isoperimetry~\cite{BogeDuzaarFusco}. We obtain qualitative stability
for arbitrary measurable projective sets: at fixed measure, convergence
to the sharp neighborhood bound forces convergence in measure, after rotations,
to a cap.

Extending the comparison to arbitrary measurable sets requires more
than convexification. A great hypersphere may separate two components
without meeting either, while meeting their convex hull; taking the
hull can therefore increase the neighborhood measure being minimized.

Spherical rearrangement inequalities provide another natural comparison
with caps.
For a bounded nonnegative kernel that is nondecreasing in the inner
product, the Baernstein--Taylor inequality \cite{BaernsteinTaylor}
maximizes the integral over two sets by replacing them with caps of the
same respective measures and a common center. Burchard and
Hajaiej~\cite{BurchardHajaiej} extended this theory to certain
nonlinear integral functionals and analyzed their equality cases. These comparisons depend on monotonicity of the kernels
in spherical distance. Orthogonality instead singles out inner product
zero: the indicators of thin belts \(|x\cdot y|<\varepsilon\)
are not monotone in \(x\cdot y\). Thus these rearrangement theorems
do not directly supply the neighborhood inequality.

The extension to measurable sets addresses Witsenhausen's 1974 problem
of maximizing the measure of a spherical set containing no orthogonal
pair \cite{Witsenhausen}. Write
\begin{equation}\label{eq:alpha}
 \alpha_n=\sup\{\sigma(E):E\subset S^{n-1}\text{ measurable},
                 \ x\cdot y\ne0\text{ for all }x,y\in E\}.
\end{equation}
We call such sets \emph{independent}. Witsenhausen proved
\(\alpha_n\le1/n\): an orthonormal frame meets an independent set
in at most one point, and averaging over rotations gives
\(n\sigma(E)\le1\). Two opposite open caps of angular radius
\(\pi/4\) supply a lower bound. Inner products are positive within
either cap and negative between the two, so their union is independent
and has measure
\begin{equation}\label{eq:delta}
 \delta_n=A_n(\pi/4)
          =c_n\int_0^{\pi/4}\sin^{n-2}\theta\,d\theta.
\end{equation}
The double-cap conjecture, attributed to Kalai, asserts that
\(\alpha_n=\delta_n\); see \cite{Kalai}.

Witsenhausen's averaging argument illustrates a general bridge between
finite configurations and measurable sets. For any finite configuration
on the sphere, averaging over its rotations bounds \(\alpha_n\) by
the largest proportion of its points that can contain no orthogonal
pair. The intersection theorem of Frankl and Wilson~\cite{FranklWilson}
provides configurations that give exponentially decreasing upper bounds
in high dimensions. Raigorodskii~\cite{Raigorodskii} improved these
bounds using vectors with coordinates in \(\{-1,0,1\}\). The
polynomial method behind these constructions and its relation to Borsuk's
problem are discussed in Kalai's survey~\cite{Kalai}.

A complementary approach uses harmonic analysis. For a measurable
set \(E\) with indicator \(\one_E\), averaging
\(\one_E(Rx)\one_E(Ry)\) over rotations \(R\) produces a continuous
positive definite kernel depending only on \(x\cdot y\).
Schoenberg~\cite{Schoenberg} characterized continuous kernels of this
kind by expansions in Gegenbauer polynomials with nonnegative
coefficients. This positivity is the basis of the linear programming
method of Delsarte, Goethals, and Seidel~\cite{DelsarteGoethalsSeidel}
for spherical codes: finite configurations with prescribed restrictions
on their pairwise angles. For measurable sets avoiding
an angle, the same method turns the forbidden inner product into a
constraint on the correlation kernel. DeCorte and
Pikhurko~\cite{DeCortePikhurko} combined this constraint with
combinatorial restrictions to improve Witsenhausen's bound on \(S^2\)
from \(1/3\) to \(\alpha_3<0.313\).

Stronger optimization methods retain more of the restrictions imposed
by an actual set. In finite graph theory, Lov\'asz's theta
function~\cite{Lovasz} gives a semidefinite bound for independence.
For spherical codes, Bachoc and Vallentin~\cite{BachocVallentin}
strengthened linear programming by using matrix positivity conditions
associated with triples of points. In the measurable setting, DeCorte,
de Oliveira Filho, and Vallentin~\cite{DeCorteOliveiraVallentin}
gave an exact convex optimization formulation using completely positive
kernels, a stronger condition than positive definiteness.
Bekker et al.~\cite{BekkerEtAl} developed optimization hierarchies
for distance-avoiding sets in compact spaces. On \(S^2\), where
\(\delta_3=1-1/\sqrt2=0.292893\ldots\), these approaches have
progressively narrowed the gap above the cap construction. The May 2026
preprint of Czifra et al.~\cite{CzifraEtAl} reports the bound
\(\alpha_3\le0.2953\), combining harmonic analysis with bounds from
finite graphs.

Theorem~\ref{thm:shadow} gives the conjectured value directly. If
\(E\) is independent, its antipodal saturation \(E\cup(-E)\) is
still independent. Its projective image \(A\) has measure
\(p=\mu(A)\ge\sigma(E)\) and is disjoint from \(N(A)\), whence
\[
 p+g_n(p)\le1.
\]
Since \(p\mapsto p+g_n(p)\) is strictly increasing and
\(\delta_n+g_n(\delta_n)=1\) by \eqref{eq:complement}, we obtain
\(\sigma(E)\le\delta_n\). Opposite open caps attain this bound.

The same extremal measure governs a fractional coloring parameter.
The \emph{spherical orthogonality graph} \(\mathcal G_n\) has vertex
set \(S^{n-1}\), with distinct vertices adjacent exactly when they
are orthogonal. Its measurable independent sets are precisely those in
\eqref{eq:alpha}. Identifying antipodal vertices gives the
\emph{projective orthogonality graph} \(\mathcal P_n\) on \(M\),
where \([x]\) and \([y]\) are adjacent when \(x\cdot y=0\).
In this language, \(N(A)\) is the set of all neighbors of points of
\(A\).

The measurable chromatic number \(\chi_{\mathrm m}(G)\) is the least
number of measurable independent sets in a partition of the vertices of
\(G\). A measurable fractional coloring allows a finite family of such
sets to overlap, assigning nonnegative weights whose sum at each vertex
is at least one. The infimum of the total weight is the measurable
fractional chromatic number \(\chi_{f,\mathrm m}(G)\). We obtain
\[
 \chi_{f,\mathrm m}(\mathcal G_n)=\alpha_n^{-1}=\delta_n^{-1},
\]
which equals \(2+\sqrt2\) for \(S^2\). For finite transitive graphs,
this reciprocal relation is classical~\cite{ScheinermanUllman};
Proposition~\ref{prop:homogeneous} gives the measurable version for a
class of compact homogeneous graphs.

Measurable colorings also connect the problem with Euclidean distance
geometry. Falconer~\cite{Falconer} studied partitions of Euclidean
space into measurable sets avoiding a fixed distance, using measure
to bound the number of colors. Bachoc, Nebe, de Oliveira Filho, and
Vallentin~\cite{BachocEtAl} developed harmonic bounds for measurable
chromatic numbers in a framework that includes the sphere.
For \(\mathcal G_n\), combining our independence bound with classical
spherical coverings \cite{RogersCovering,Prosanov} determines the
exponential growth rate:
\(\chi_{\mathrm m}(\mathcal G_n)=(\sqrt2+o(1))^n\)
(Corollary~\ref{cor:colorings}).

Orthogonality also gives the problem a connection to quantum theory.
Orthogonal state vectors represent states that a suitable measurement
can perfectly distinguish. The Kochen--Specker theorem~\cite{KochenSpecker}
shows that, already in dimension three, it is impossible to assign
values in \(\{0,1\}\) to all rays (one-dimensional subspaces) so
that each orthonormal basis
contains exactly one ray assigned the value \(1\). Such an assignment
would select a predetermined outcome in every basis, with the value of
a ray independent of the basis in which it is measured. A more direct connection to the
extremal measure problem was established by Montina~\cite{Montina},
who used the double-cap conjecture to obtain conditional lower bounds
on the classical communication needed to simulate quantum state
transmission and measurement.

The proof determines the optimal neighborhood profile, which assigns
to each \(p\) the least possible neighborhood measure of a set of
measure \(p\). This optimization can be expressed in terms of pairs
of sets with no orthogonal pair between them. The Funk
transform~\cite{Funk}, which averages functions over great hyperspheres,
records the incidence between the two sets. We use its compactness to
obtain optimal pairs, adapting the method of DeCorte and
Pikhurko~\cite{DeCortePikhurko} for the existence of maximum-measure
angle-avoiding sets when \(n\ge3\).

To identify the profile, we follow the variational comparison approach
used by Morgan and Johnson~\cite{MorganJohnson} and Bayle and
Rosales~\cite{BayleRosales} for isoperimetric profiles. Rotation
comparisons first show that the optimal profile is concave. Paired
projective transformations then allow us to deform the optimal pairs
while preserving their orthogonality exclusion. A key step is to show
that, at almost every prescribed measure, an optimizer admits a
deformation that changes its measure to first order. This allows us
to extract a differential inequality for the profile from the second
variations. Integrating the comparison with the cap equation identifies
the profile, and the equality conditions force the extremizers to agree
almost everywhere with caps.

\paragraph{Separate Lean proof.}
The \emph{LeanShadow} repository contains a machine-checked Lean 4
proof of the sharp neighborhood inequality and the double-cap result, including
their rigidity statements (Theorem~\ref{thm:shadow} and
Corollary~\ref{cor:double}):
\begin{center}
\url{https://github.com/allenibiapina/LeanShadow}
\end{center}
This development uses a different proof from the one presented in this
paper. The formalization and verification workflow were carried out almost
entirely using AI-assisted tools. Its formal verification covers the
statements and arguments in the repository; it does not constitute a
verification of the proof or the additional applications presented here.
The repository gives the exact formal statements and instructions for
reproducing the checks.

\section{Preliminaries and construction of the proof}

\subsection{Measure, incidence, and the optimal profile}\label{sec:measure-profile}
Throughout the proof \(n\ge3\) is fixed. A Borel set belongs to the
sigma-algebra generated by the open sets. The neighborhood of a Borel set
\(A\) is the projection onto the second coordinate of
\(\{(x,y)\in M\times M:x\in A,\ x\perp y\}\).
Projections of Borel sets in a product of compact metric spaces are
measurable for every completed Borel probability measure \cite{Kechris}.
Thus \(N(A)\) is \(\mu\)-measurable. If \(A\) is compact, so is \(N(A)\).
For measurable \(D\), its density-one representative is the Borel set
\[
 D^*=\left\{x:\lim_{r\downarrow0}
             \frac{\mu(D\cap B_r(x))}{\mu(B_r(x))}=1\right\},
\]
where \(B_r(x)\) is the projective metric ball. The Lebesgue density
theorem gives \(\mu(D\triangle D^*)=0\); here \(\triangle\) denotes
symmetric difference. This is the indicator-function case of the Lebesgue
differentiation theorem \cite{EvansGariepy}, transferred to Riemannian
volume in local coordinates. A density-zero point of \(D\) is a
density-one point of its complement.
Changing a set on a null subset can nevertheless change its neighborhood.
For example, a projective hyperplane \(H=z^\perp\) has measure zero,
but \(N(H)=M\) when \(n\ge3\): every \(y^\perp\) meets \(z^\perp\).
This is why pointwise avoidance and the choice of representatives
must remain explicit.

Two sets \(A,B\subset M\) form an \emph{avoiding pair}, or
\emph{avoid one another}, if \(x\cdot y\ne0\) for every
\(x\in A\) and \(y\in B\). Equivalently,
\(B\cap N(A)=\varnothing\). This condition places no restriction
on pairs within either set. An independent set \(A\) is exactly
a set for which \((A,A)\) is an avoiding pair.

Define
\begin{equation}\label{eq:f}
 f(p)=\inf\{\mu(N(A)): A\subset M\text{ Borel},\ \mu(A)=p\}.
\end{equation}
For a fixed \(A\), the largest set avoiding it is \(M\setminus N(A)\).
Thus \(1-f(p)\) is the supremal measure of a set avoiding some set
of measure \(p\). This reformulates the minimization problem as an
optimization over two sets. A Borel set \(A\) of measure \(p\)
with \(\mu(N(A))=f(p)\) is an \emph{optimizer at measure \(p\)}.
If a Borel set \(B\) avoids this \(A\) and has measure \(1-f(p)\),
we call \((A,B)\) an \emph{optimal avoiding pair at measure \(p\)}.
Proposition~\ref{prop:attainment} proves that such pairs exist.
The aim is to identify \(f=g_n\) and classify its optimizers
up to null sets.

Write \(\one_E\) for the indicator of \(E\), and use the real
\(L^2\) inner product \(\ip{h}{k}=\int_M hk\,d\mu\).
The probability-normalized Funk transform is
\begin{equation}\label{eq:F}
 Fh(x)=\int_{x^\perp}h(y)\,d\mu_{x^\perp}(y),
\end{equation}
where the fiber \(x^\perp\cong\RP^{n-2}\) carries normalized volume.
For an indicator, \(F\one_E(x)\) is the fraction of directions
perpendicular to \(x\) that belong to \(E\).
Let \(SO(n)\) be the rotation group, equipped with its invariant
probability measure, called normalized Haar measure and denoted \(dR\).
The \emph{orthogonality incidence measure} \(\nu\) is the law
of \(([Re_1],[Re_2])\) for a rotation sampled from this measure.
Equivalently, first sample \(x\) uniformly in \(M\)
and then \(y\) uniformly in \(x^\perp\). Its marginals are \(\mu\),
it is symmetric under exchange of coordinates, and
\(\ip{h}{Fk}=\int h(x)k(y)\,d\nu(x,y)\).
Consequently \(F\) extends to a self-adjoint operator on \(L^2(M)\)
with \(\|Fh\|_2\le\|h\|_2\). It preserves nonnegative functions, with \(F1=1\) and
\(\int Fh=\int h\).
These statements concern equivalence classes modulo null sets;
individual fiber integrals can depend on the representative.
The \emph{incidence} between measurable sets \(D,E\) is
\(\nu(D\times E)=\ip{\one_D}{F\one_E}\); zero incidence means
that this quantity vanishes. Avoiding sets have zero incidence.
Lemma~\ref{lem:compact} recovers pointwise avoidance from zero
incidence by passing to density-one representatives.

For any fixed unit \(z\), the coordinate \(t=|x\cdot z|\) has density
\begin{equation}\label{eq:latitude}
 w_n(t)=c_n(1-t^2)^{(n-3)/2},\quad 0<t<1.
\end{equation}
Every latitude is null, so the measure of
\(E\cap\{|x\cdot z|\le t\}\) varies continuously from \(0\) to
\(\mu(E)\) as \(t\) runs from \(0\) to \(1\). The intermediate value
theorem therefore trims a Borel set to any smaller prescribed measure.
We shall use this exact trimming repeatedly.

\subsection{Deformations, moments, and concavity}
Paired projective deformations let us vary the measures of fixed
sets while preserving avoidance. Let \(SL(n,\R)\) denote the real
matrices of determinant one. For \(T\in SL(n,\R)\), the notation
\(TE\) means the image under \(x\mapsto Tx/\|Tx\|\), and
\(T^{-T}=(T^{-1})^T\). The paired transformation
\((A,B)\mapsto(TA,T^{-T}B)\) preserves avoidance because
\((Tx)\cdot(T^{-T}y)=x\cdot y\); normalization does not change
whether this inner product vanishes.

For real symmetric matrices \(H,K\) of the same size, we use the
Frobenius inner product \(\ip{H}{K}=\tr(HK)\) and norm
\(\|H\|_F=\sqrt{\ip{H}{H}}\). The symmetric deformation directions
form the space \(V=\Sym_0(n)\) of symmetric trace-free
\(n\times n\) matrices, of dimension \(d=(n-1)(n+2)/2\).
For a self-adjoint operator \(Q\) on \(V\), \(Q\succeq0\) means
\(\ip{QH}{H}\ge0\) for every \(H\in V\); for a symmetric
\(n\times n\) matrix, it means \(z^TQz\ge0\) for every \(z\in\R^n\).
\(Q\preceq R\) means \(R-Q\succeq0\). We write \(\tr_V\) for the
trace of an operator acting on \(V\).
For a measurable set \(E\), put
\begin{equation}\label{eq:moments}
 S_E=\int_E xx^T\,d\mu(x),\qquad G_E=\mu(E)I-nS_E\in V.
\end{equation}
The condition \(G_E=0\) is called isotropy. It means that all
directional second moments are equal. The matrix \(G_E\) will also
represent the first variation of the measure under projective
deformations.

For the scalar comparison, we use the regularity of concave functions.
A continuous concave function \(f:[0,1]\to\R\) has a right derivative
\(v=f'_+\) that is right-continuous, nonincreasing, and locally bounded
in \((0,1)\). The function \(f\) is locally absolutely continuous,
and \(v\) is differentiable almost everywhere. We write \(f''=v'\)
where this derivative exists. Integrating the first-order expansion of
\(v\) gives the second-order Peano expansion at almost every interior \(p\):
\begin{equation}\label{eq:Peano}
 f(p+h)=f(p)+f'(p)h+\tfrac12 f''(p)h^2+o(h^2),
 \qquad f''(p)\le0.
\end{equation}
We call such \(p\) \emph{second-order points}. This expansion at a single
point does not assert that \(f\) is \(C^2\); see \cite{EvansGariepy}.

The right derivative need not itself be absolutely continuous.
To integrate the differential comparison, we will use the following
inequality for monotone functions \cite[Theorem 4.9]{Ryzhik}:
if \(w\) is right-continuous and nonincreasing on an open interval, then
\begin{equation}\label{eq:monotone-integral}
 w(t)-w(s)\le\int_s^t w'(x)\,dx
 \qquad(s<t)
\end{equation}
on every compact subinterval. Indeed, Fatou's lemma applied to
the nonnegative quotients \((w(x)-w(x+h))/h\), \(h\downarrow0\),
gives \(\int_s^t(-w')\le w(s)-w(t)\).
The inequality accounts for both jumps and singular continuous decreases.

\subsection{Proof outline}
Caps give \(f\le g_n\); we need the reverse inequality and the
classification of optimizers. Concavity of \(f\), which we establish
first, gives only the chord bound \(f(p)\ge p\) from the endpoint
values. The main idea is to prove concavity
of \(\Phi=f\circ g_n^{-1}\). In the coordinate \(s=g_n(p)\), the
cap profile becomes the line \(s\mapsto s\). Since both profiles
have endpoint values \(0\) and \(1\), concavity places \(\Phi\)
above this line, giving \(f\ge g_n\). The proof develops the
variations needed for this comparison as follows.

\emph{Attainment (\S~\ref{sec:attainment}).}
Compactness of the Funk transform preserves zero incidence under
weak limits. Although the limits need not be indicators, their positive
supports have enough measure to recover an optimal pair. Passing to
density-one representatives restores pointwise avoidance, and trimming
the sets gives the prescribed measures. These fixed sets are the starting
point for the later variations.

\emph{Cap comparison and concavity (\S~\ref{sec:concavity}).}
We compute cap neighborhoods and establish \(p\le f(p)\le g_n(p)\),
with endpoint values \(f(0)=0\) and \(f(1)=1\). Adding small pieces
of rotated caps proves continuity of \(f\). We then compare a set
with a rotated copy. Their union and intersection have measures
\(p+h\) and \(p-h\), while their neighborhood measures sum to
at most twice the original neighborhood measure. Rotations realize
every sufficiently small \(h\), giving the midpoint inequality
and hence concavity of \(f\). Together with the cap bound, it gives a positive,
nonincreasing right derivative \(f'_+\) and second-order expansions
almost everywhere. These are the regularity properties needed to
formulate and integrate the differential comparison.

\emph{Projective variations (\S~\ref{sec:variations}).}
For an optimal pair \((A,B)\), deform the sets to
\((TA,T^{-T}B)\) and write \(a(T)=\mu(TA)\),
\(b(T)=\mu(T^{-T}B)\). Avoidance is preserved, so
\[
 b(T)+f(a(T))\le1,
 \qquad\text{with equality at }T=I.
\]
The projective Jacobian makes \(a\) and \(b\) analytic without
boundary assumptions on the sets. We compute their derivatives and
establish the rotation averages and nonconstancy used below. Each
variation starts from one fixed optimal pair; the transformed sets
need not remain optimal. They provide competitors at the new measure
\(a(T)\), which is why \(f(a(T))\) enters the comparison. The measure
functions are analytic in \(T\), whereas the profile needs only its
second-order expansion at the original measure \(p\).

\emph{Contraction identity (\S~\ref{sec:contraction}).}
The first measure derivatives are represented by the moment matrices
\(G_E\). We construct a linear map \(\Psi\) that preserves
positive semidefiniteness and sends each operator representing second
measure derivatives to a multiple of \(G_E\). At a second-order point,
the optimality expression has a maximum at \(I\), so the operator
representing its negative second variation is positive semidefinite.
Applying \(\Psi\) preserves this positivity and expresses the resulting
inequality in terms of \(G_A\) and \(G_B\). The first-order optimality
relation then eliminates \(G_B\). The same contraction identity is
also used to exclude isotropic optimizers.

\emph{Nonisotropy (\S~\ref{sec:nonisotropy}).}
At every second-order point of \(f\), we prove that every optimizer
satisfies \(G_A\ne0\). Thus some projective deformation changes
its measure to first order. This is necessary because otherwise the
term involving \(f''\) in the second variation would vanish.
Suppose \(G_A=0\). Analyticity and nonconstancy give a first nonzero
Taylor term of degree \(m\ge2\) in the measure variation. Optimality
gives an upper bound of order \(2m\) for the signed displacement
from the tangent line to \(b=1-f(a)\) at the original pair. This
displacement has zero average over deformations conjugated by rotations,
so its terms of degree below \(2m\) vanish. Along a direction where
the degree-\(m\) term is nonzero, the contraction identity bounds
the measure derivative by order \(2m-2\). Integration gives order
\(2m-1>m\), a contradiction.

\emph{Differential comparison (\S~\ref{sec:differential}).}
Applying \(\Psi\) to second-order optimality gives a matrix
inequality involving \(G_A\). Since \(G_A\ne0\) is trace-free,
it has a negative eigenvalue, which yields a scalar bound on
\(f''/f'\). Among sets of prescribed measure, a cap maximizes
the second moment in its axial direction. For a fixed unit axis \(z\),
this follows by placing the prescribed measure where \((x\cdot z)^2\)
is largest. The level sets of this function have measure zero, so equality
determines the cap up to a null set. This sharp moment bound makes the
scalar estimate precisely
\[
 \frac{f''(p)}{f'(p)}\le\frac{g_n''(p)}{g_n'(p)}
 \qquad\text{at every second-order point }p.
\]

\emph{Identification of the profile and rigidity (\S~\ref{sec:comparison}).}
Monotonicity of \(f'_+\) lets us integrate the differential
comparison, including possible jumps and singular decreases, to show
that \(f'_+/g_n'\) is nonincreasing. In the coordinate
\(s=g_n(p)\), this ratio is the right derivative of \(\Phi\),
so \(\Phi\) is concave and the endpoint argument gives \(f=g_n\).
The differential comparison is initially available only almost everywhere.
Once the profiles agree, the smoothness of \(g_n\) makes every interior
measure a second-order point of \(f\). We can therefore return to the
variational inequalities for any optimizer, and equality in the moment
bound forces it to agree almost everywhere with a cap. Finally, for a
completed-measurable set, choose a Borel subset of full measure. Its
neighborhood is contained in the original neighborhood, so monotonicity
of inner measure transfers the lower bound. If equality holds, the Borel
subset is itself an optimizer, and its rigidity transfers to the original
set.

Figure~\ref{fig:proof} follows the central deduction. The arrow labels
indicate the inputs needed to pass from one conclusion to the next.

\begin{figure}[!t]
\centering
\begin{tikzpicture}[
 box/.style={draw=black!65,rounded corners=2pt,align=center,
             font=\small,text width=110mm,inner sep=6pt},
 reason/.style={font=\footnotesize,align=left,fill=white,inner sep=2pt},
 arr/.style={-{Stealth[length=2mm]},thick},
 node distance=18mm]
 \node[box] (optimality) {\textbf{Optimality under projective deformations}\\[2pt]
   \(b(T)+f(a(T))\le1\), with equality at \(I\)};
 \node[box,below=of optimality] (matrix) {\textbf{A matrix inequality}\\[2pt]
   \(-f''(p)G_A^2+(n-2)f'(p)G_A\succeq0\)};
 \node[box,below=of matrix] (differential) {\textbf{The sharp differential comparison}\\[2pt]
   \(f''(p)/f'(p)\le g_n''(p)/g_n'(p)\)};
 \node[box,below=of differential] (comparison) {\textbf{Identification of the profile}\\[2pt]
   \(f\circ g_n^{-1}\text{ concave}\quad\Longrightarrow\quad f=g_n\)};
 \draw[arr] (optimality)--node[reason,right=3pt]
   {Second variations and positive contraction\\
    \S\S~\ref{sec:variations}--\ref{sec:contraction}, \ref{sec:differential}} (matrix);
 \draw[arr] (matrix)--node[reason,right=3pt]
   {\(G_A\ne0\) and the cap moment bound\\
    \S\S~\ref{sec:nonisotropy}--\ref{sec:differential}} (differential);
 \draw[arr] (differential)--node[reason,right=3pt]
   {Monotone-slope integration and endpoint comparison\\
    \S~\ref{sec:concavity}, \S~\ref{sec:comparison}} (comparison);
\end{tikzpicture}
\caption{From optimality to the cap profile. The first three boxes concern
an optimal pair at a second-order point \(p\). The last step uses the
comparison almost everywhere and the cap bound \(f\le g_n\).
Once \(f=g_n\), equality in the moment bound gives rigidity
(\S~\ref{sec:comparison}).}
\label{fig:proof}
\end{figure}
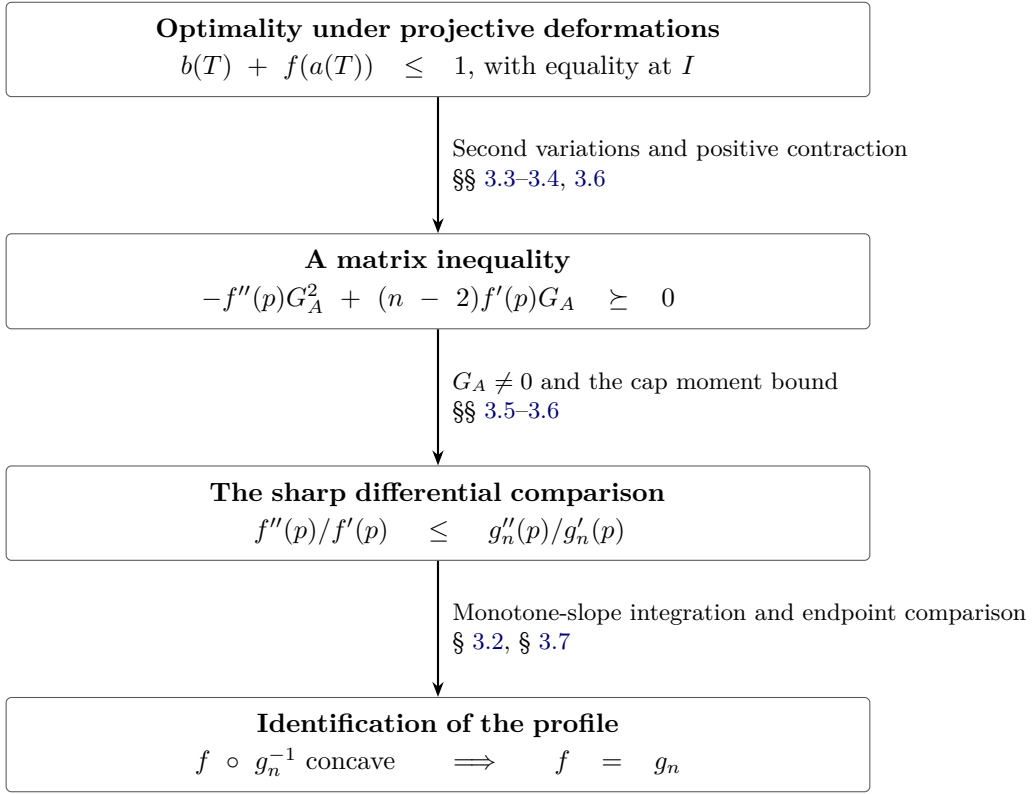

\section{Proof of the neighborhood inequality}

An optimal avoiding pair turns the inequality
\(\mu(B)+f(\mu(A))\le1\) into an equality, so admissible deformations
constrain the first and second measure variations. To use this
observation, we first prove that optimal pairs exist and that \(f\)
admits second-order expansions almost everywhere. Compactness and
rotation averaging provide these two ingredients.

\subsection{Compactness of the Funk transform and attainment}\label{sec:attainment}
We first show that the infimum defining \(f\) is attained, so that
the later variations can start from a fixed optimal pair. Weak limits
of set indicators may take values strictly between zero and one.
Compactness of \(F\) preserves zero incidence under these limits,
and density-one representatives recover sets that avoid one another.

\begin{lemma}\label{lem:compact}
For \(n\ge3\), the operator \(F\) is compact on \(L^2(M)\).
Moreover, if \(\ip{\one_D}{F\one_E}=0\), then \(D^*\) and \(E^*\)
avoid one another.
\end{lemma}
\begin{proof}
Identify \(L^2(M)\) with the even subspace of \(L^2(S^{n-1})\).
Then \(F\) is the restriction of the spherical averaging operator
\(A_0\), which is compact by
\cite[Corollary~8.3]{DeCortePikhurko}.

For the second assertion, let \(\widetilde D,\widetilde E\) be the
antipodal lifts of \(D,E\), and suppose they have orthogonal
density-one points \(x,y\), respectively. Choose disjoint spherical
caps \(U,V\) centered at \(x,y\), of radius less than \(\pi/4\),
and put \(I=(\widetilde D\cap U)\cup(\widetilde E\cap V)\).
Neither cap contains an orthogonal pair, and the incidence between
\(\widetilde D\) and \(\widetilde E\) is zero. Hence
\(\langle\one_I,A_0\one_I\rangle_{L^2(S^{n-1})}=0\).
But \(x,y\) are density-one points of \(I\), so
\cite[Lemma~8.5]{DeCortePikhurko} makes this integral positive,
a contradiction.
\end{proof}

\Needspace{9\baselineskip}
\begin{proposition}\label{prop:attainment}
For every \(p\in[0,1]\), there are Borel sets \(A,B\subset M\)
that avoid one another, meaning \(x\cdot y\ne0\) for every
\(x\in A\) and \(y\in B\), and satisfy
\[
 \mu(A)=p,\qquad \mu(B)=1-f(p),\qquad \mu(N(A))=f(p).
\]
\end{proposition}
\begin{proof}
Choose Borel \(A_j\) of measure \(p\) with neighborhood measures tending
to \(f(p)\), and full-measure Borel subsets
\(B_j\subset M\setminus N(A_j)\).
Pass to common weak \(L^2\) limits
\(\one_{A_j}\rightharpoonup u\), \(\one_{B_j}\rightharpoonup v\).
Then \(0\le u,v\le1\), \(\int u=p\), and \(\int v=1-f(p)\).
Compactness gives
\[
 \ip{u}{Fv}=\lim_j\ip{\one_{A_j}}{F\one_{B_j}}=0.
\]
Nonnegativity, applied first to positive level sets and then to
their union, implies zero incidence between
\(D=\{u>0\}\) and \(E=\{v>0\}\).
Their measures are at least \(p\) and \(1-f(p)\).
By Lemma~\ref{lem:compact}, their density-one representatives avoid
one another. Trim those representatives to Borel sets \(A,B\) of the
required measures. Avoidance gives \(\mu(N(A))\le f(p)\), and the
definition of \(f\) gives the reverse inequality.
\end{proof}
This is the compactness method of
\cite[proof of Theorem~8.6]{DeCortePikhurko}, applied to two sets
with one prescribed measure.

\subsection{Cap comparison and concavity}\label{sec:cap-continuity}
\label{sec:concavity}
We first establish elementary bounds and continuity of \(f\).
The continuity argument is a warm-up to the proof of concavity:
averaging over rotations finds a cap that adds enough measure to
a given set, while the cap formula controls the increase of its
neighborhood.

\begin{proposition}\label{prop:cap-continuity}
The profile \(f\) is nondecreasing and continuous on \([0,1]\), with
\begin{equation}\label{eq:elementary}
 p\le f(p)\le g_n(p),\qquad f(0)=0,\quad f(1)=1.
\end{equation}
More precisely, for \(0\le p<1\) and \(0<h\le1-p\),
\begin{equation}\label{eq:continuity}
 0\le f(p+h)-f(p)\le g_n\!\left(\frac{h}{1-p}\right).
\end{equation}
\end{proposition}
\begin{proof}
The neighborhood of the cap centered at \(z\) with radius
\(0\le r\le\pi/2\) is
\begin{equation}\label{eq:cap-shadow}
 N(\cC_r(z))=\{y:|y\cdot z|\le\sin r\}.
\end{equation}
Integrating \eqref{eq:latitude} proves \eqref{eq:angular} and
\eqref{eq:g}, and hence \(f(p)\le g_n(p)\). For
\(0<r<\pi/2\), open caps give open belts with the same measures.
The lower bound \(p\le f(p)\) follows from
\(0\le F\one_A\le\one_{N(A)}\) almost everywhere and
\(\int F\one_A=\mu(A)\). These bounds also give the endpoint
values in \eqref{eq:elementary}, and exact trimming makes \(f\)
nondecreasing.

To estimate the cost of increasing the measure from \(p\) to
\(p+h\), fix a Borel set \(A\) of measure \(p<1\) and a cap
\(C\) of measure \(h/(1-p)\). By Fubini,
\[
 \int_{SO(n)}\mu(RC\setminus A)\,dR=(1-p)\mu(C)=h.
\]
Thus some rotation satisfies \(\mu(RC\setminus A)\ge h\).
Trim \(RC\setminus A\) to a Borel set \(B\) of measure \(h\).
Then \(\mu(A\cup B)=p+h\) and
\(N(A\cup B)\subset N(A)\cup RN(C)\), so
\(f(p+h)\le\mu(N(A))+g_n(h/(1-p))\).
Taking the infimum over \(A\) proves \eqref{eq:continuity}.
Since \(g_n(t)\to0\) as \(t\downarrow0\), this estimate gives
right continuity; applying it at \(p-h\) gives left continuity
in the interior. Finally, \eqref{eq:elementary} gives continuity
at both endpoints.
\end{proof}

Rotation averaging also gives concavity, which supplies second-order
expansions and a monotone slope for the later comparison. We now
compare a set with a rotated copy: submodularity bounds the sum of
the neighborhood measures of their union and intersection, yielding
the midpoint inequality for \(f\).

\begin{proposition}\label{prop:concavity}
The profile \(f\) is concave on \([0,1]\), and \(f'_+(p)>0\) for
\(0<p<1\).
\end{proposition}
\begin{proof}
For Borel sets \(C,D\subset M\), we have
\(N(C\cup D)=N(C)\cup N(D)\) and
\(N(C\cap D)\subset N(C)\cap N(D)\).
Taking measures gives the submodular inequality
\begin{equation}\label{eq:submodularity}
 \mu(N(C\cup D))+\mu(N(C\cap D))
 \le\mu(N(C))+\mu(N(D)).
\end{equation}

Fix \(p\in(0,1)\) and any Borel set \(A\) of measure \(p\).
For \(R\in SO(n)\), put \(h_A(R)=\mu(A\setminus RA)\).
This is continuous: approximate \(\one_A\) in \(L^1\) by continuous
functions and use rotation invariance of \(\mu\).
Moreover, \(h_A(I)=0\), and Fubini gives
\begin{equation}\label{eq:overlap-average}
 \int_{SO(n)}h_A(R)\,dR
 =p-\int_{SO(n)}\mu(A\cap RA)\,dR
 =p(1-p).
\end{equation}
Since \(SO(n)\) is compact and connected, the image of \(h_A\)
is an interval containing both \(0\) and a value at least \(p(1-p)\).
Thus every \(h\in[0,p(1-p)]\) is realized by a rotation.
For such a rotation,
\(\mu(A\cap RA)=p-h\) and \(\mu(A\cup RA)=p+h\).
Apply \eqref{eq:submodularity} with \(C=A,D=RA\), and use
\(N(RA)=RN(A)\), to obtain
\[
 f(p-h)+f(p+h)
 \le\mu(N(A\cap RA))+\mu(N(A\cup RA))
 \le2\mu(N(A)).
\]
Taking the infimum over all \(A\) of measure \(p\) yields
\begin{equation}\label{eq:midpoint}
 f(p-h)+f(p+h)\le2f(p),
 \qquad 0\le h\le p(1-p).
\end{equation}
On any compact subinterval of \((0,1)\), the admissible upper bound
for \(h\) is uniformly positive. Hence sufficiently fine equally
spaced grids satisfy \eqref{eq:midpoint} at every interior grid
point. Their successive secant slopes are nonincreasing, so their
values lie above the endpoint chord. The continuity established in
Proposition~\ref{prop:cap-continuity}, first along these grids and
then at \(0,1\), proves concavity on \([0,1]\).

Finally, \(f'_+(p)\ge\frac{1-f(p)}{1-p}>0\), since
\(f(p)\le g_n(p)<1\) in the interior.
\end{proof}

\subsection{Projective measure variations}\label{sec:variations}
For fixed Borel sets \(A,B\subset M\), consider
\((A,B)\mapsto(TA,T^{-T}B)\), with \(T\in SL(n,\R)\), and write
\begin{equation}\label{eq:ab}
 a(T)=\mu(TA),\qquad b(T)=\mu(T^{-T}B).
\end{equation}
The inverse-transpose action on \(B\) is chosen so that an avoiding
pair remains avoiding. We now establish the properties of these
deformations that will be used in the rest of the proof.

\begin{lemma}\label{lem:feasibility}
If \((A,B)\) is an avoiding pair, then
\begin{equation}\label{eq:feasible}
 b(T)+f(a(T))\le1
\end{equation}
for every \(T\in SL(n,\R)\), with equality at \(T=I\) whenever
\((A,B)\) is optimal.
\end{lemma}
\begin{proof}
The identity \(N(TA)=T^{-T}N(A)\) shows that \((TA,T^{-T}B)\)
is also an avoiding pair. The definition of \(f\) therefore gives
\eqref{eq:feasible}. Optimality means \(\mu(B)=1-f(\mu(A))\), which
gives equality at \(T=I\).
\end{proof}

\Needspace{9\baselineskip}
Rotations enter the argument once more through the following two
properties of the deformed measures. The first follows from rotation
invariance of \(\mu\) and the identities \((RT)A=R(TA)\) and
\((RT)^{-T}B=R(T^{-T}B)\) for \(R\in SO(n)\).

\begin{lemma}\label{lem:rotation-invariance}
For every \(R\in SO(n)\) and \(T\in SL(n,\R)\),
\(a(RT)=a(T)\) and \(b(RT)=b(T)\).
\end{lemma}

The second property concerns the conjugates \(RTR^{-1}\), which
change the orientation of the deformation. To average their measures,
we use the projective Jacobian
\begin{equation}\label{eq:Jacobian}
 J_T(x)=\|Tx\|^{-n}.
\end{equation}
This is also the classical Poisson kernel for real projective boundary
data; see Veech~\cite[\S8, equation (8.12)]{Veech}, with \(g=T^{-1}\).
Change of variables gives
\begin{equation}\label{eq:measure-jacobians}
 a(T)=\int_A J_T(x)\,d\mu(x),\qquad
 b(T)=\int_B J_{T^{-T}}(x)\,d\mu(x).
\end{equation}

\begin{lemma}\label{lem:rotation-average}
For every \(T\in SL(n,\R)\),
\begin{equation}\label{eq:rotation-average}
 \int_{SO(n)}a(RTR^{-1})\,dR=\mu(A),\qquad
 \int_{SO(n)}b(RTR^{-1})\,dR=\mu(B),
\end{equation}
where \(dR\) is normalized Haar measure.
\end{lemma}
\begin{proof}
Rotations have Jacobian one, so
\(J_{RTR^{-1}}(x)=J_T(R^{-1}x)\). Hence, for every unit \(x\),
\[
 \int_{SO(n)}J_{RTR^{-1}}(x)\,dR
 =\int_M J_T(y)\,d\mu(y)=1.
\]
The last equality is change of variables on \(M\). Integrating over
\(A\) proves the first identity in \eqref{eq:rotation-average}; using
\(T^{-T}\) and integrating over \(B\) proves the second.
\end{proof}

The same change-of-variables formula gives regularity in \(T\): the
sets of integration stay fixed, and all dependence on the matrix is
carried by the Jacobian.

\begin{lemma}\label{lem:analytic-variations}
The functions \(a,b\) are real analytic on \(SL(n,\R)\), and their
derivatives may be computed by differentiating the projective
Jacobians under the integral. For \(H\in V\), the first two Jacobian
derivatives along \(e^{sH}\) are
\begin{equation}\label{eq:variations}
 J'_0=-n x^THx,\qquad
 J''_0=n(n+2)(x^THx)^2-2n x^TH^2x,
 \quad J_s=J_{e^{sH}}.
\end{equation}
\end{lemma}
\begin{proof}
In \eqref{eq:measure-jacobians}, both integrals are over fixed sets.
On a compact matrix neighborhood, the norms in the denominators have
a uniform positive lower bound. The local binomial expansions converge
uniformly in \(x\), as do their differentiated series on smaller
neighborhoods. Termwise integration proves analyticity and permits
differentiation under the integral, without any boundary regularity
of the sets.

For the derivative formulas, put
\(q(s)=\|e^{sH}x\|^2=x^Te^{2sH}x\). Then
\[
 q(0)=1,\qquad q'(0)=2x^THx,\qquad q''(0)=4x^TH^2x.
\]
Since \(J_s(x)=q(s)^{-n/2}\), the chain rule gives
\[
 J'_0=-\frac n2 q'(0),\qquad
 J''_0=\frac{n(n+2)}4 q'(0)^2-\frac n2 q''(0).
\]
Substitution yields \eqref{eq:variations}.
\end{proof}

The later Taylor argument requires a variation that is not locally
constant. We obtain this from uniqueness of the projective Poisson
transform.

\begin{lemma}\label{lem:nonconstancy}
If \(0<\mu(A)<1\), then \(H\mapsto a(e^H)\) is nonconstant in every
neighborhood of \(0\in V\).
\end{lemma}
\begin{proof}
Put \(p=\mu(A)\) and suppose that \(H\mapsto a(e^H)\) is constant
near zero. By Lemma~\ref{lem:analytic-variations}, it is analytic
on the connected space \(V\), so \(a(e^H)=p\) for every \(H\in V\).
Every \(T\in SL(n,\R)\) has a polar decomposition \(T=Re^H\),
with \(R\in SO(n)\) and \(H\in V\). Rotation invariance
(Lemma~\ref{lem:rotation-invariance}) therefore gives \(a(T)=p\).
By \eqref{eq:measure-jacobians} and \(\int_M J_T\,d\mu=1\),
\[
 \int_M J_T\one_A\,d\mu
 =p=\int_M J_Tp\,d\mu
 \qquad(T\in SL(n,\R)).
\]
A finite positive Borel measure on \(M\) is uniquely determined by
these Poisson integrals; see Veech~\cite[pp.~921--922, proof of
Theorem~8.7 and Remark~8.10]{Veech}, with \(g=T^{-1}\).
Hence \(\one_A\,d\mu=p\,d\mu\), so \(\one_A=p\) almost everywhere,
contrary to \(0<p<1\).
\end{proof}

\subsection{A contraction identity for second variations}\label{sec:contraction}
For a Borel set \(E\subset M\) and a deformation direction \(H\in V\),
the first and second variations of its measure are the first and second
derivatives of \(\mu(e^{sH}E)\) at \(s=0\).
Using the change-of-variables formula \eqref{eq:measure-jacobians},
we differentiate under the integral by Lemma~\ref{lem:analytic-variations}
and substitute \(J'_0\) from \eqref{eq:variations}. The identity
\(x^THx=\ip{xx^T}{H}\), the definitions in \eqref{eq:moments}, and
\(\tr H=0\) then give
\[
 \begin{aligned}
 \left.\frac{d}{ds}\mu(e^{sH}E)\right|_{s=0}
 &=\int_E J'_0(x)\,d\mu(x)\\
 &=-n\int_E x^THx\,d\mu(x)\\
 &=-n\ip{S_E}{H}\\
 &=\ip{G_E}{H}.
 \end{aligned}
\]
Thus \(G_E\) records the first measure derivative in every direction.
For \(a\) and \(b\), the first derivatives along \(e^{sH}\) are
\(\ip{G_A}{H}\) and \(-\ip{G_B}{H}\), respectively; the minus sign
comes from deforming \(B\) by \(e^{-sH}\). We will represent the
second measure derivatives by self-adjoint operators on \(V\),
constructed below from the formula for \(J''_0\) in
\eqref{eq:variations}.

For an optimal avoiding pair, \eqref{eq:feasible} gives
\(b(T)+f(a(T))\le1\), with equality at \(T=I\). At a second-order
point \(p=\mu(A)\) of \(f\), the second-order term of this expression
along \(e^{sH}\) is therefore nonpositive for every \(H\in V\).
The operator representing its negative is therefore positive
semidefinite. This is the positivity supplied by optimality.

We now construct a linear map from self-adjoint operators on \(V\)
to symmetric \(n\times n\) matrices that preserves positive
semidefiniteness. Its key feature is that it sends the operator
representing each second measure derivative to a multiple of \(G_E\).
Applying it to the optimality condition will therefore give an
inequality involving the moment matrices.

Fix an orthonormal basis \(H_1,\dots,H_d\) of \(V\).
For a self-adjoint operator \(Q:V\to V\), let
\(Q_{ij}=\ip{QH_j}{H_i}\) be its matrix coefficients in this basis,
so that \(QH_j=\sum_iQ_{ij}H_i\). Define
\begin{equation}\label{eq:Psi}
 \Psi(Q)=\sum_{i,j}Q_{ij}\frac{H_iH_j+H_jH_i}{2}.
\end{equation}
The coefficients \(Q_{ij}=Q_{ji}\) are real numbers, and
\(H_iH_j\) and \(H_jH_i\) are ordinary matrix products.
Their average is symmetric because \((H_iH_j)^T=H_jH_i\),
so \(\Psi(Q)\) is a symmetric \(n\times n\) matrix.
This construction is an instance of the standard \emph{partial trace}
from matrix analysis; see \cite[\S4.3, (4.34)]{Bhatia}.
The next lemma proves its properties directly in the present setting.

\begin{lemma}\label{lem:positive-contraction}
The map \(\Psi\) is independent of the chosen orthonormal basis of
\(V\). For every self-adjoint \(Q:V\to V\),
\(\tr\Psi(Q)=\tr_VQ\). Moreover,
\begin{equation}\label{eq:positive}
 Q\succeq0\ \Longrightarrow\ \Psi(Q)\succeq0.
\end{equation}
\end{lemma}
\begin{proof}
Choose an orthonormal eigenbasis \(T_1,\dots,T_d\) of \(Q\), with
\(QT_k=\lambda_kT_k\). Then
\(Q_{ij}=\sum_k\lambda_k\ip{T_k}{H_i}\ip{T_k}{H_j}\).
Since \(Q_{ij}=Q_{ji}\), substitution into \eqref{eq:Psi} gives
\[
 \begin{aligned}
 \Psi(Q)
 &=\sum_{i,j}Q_{ij}H_iH_j\\
 &=\sum_k\lambda_k\sum_{i,j}\ip{T_k}{H_i}\ip{T_k}{H_j}H_iH_j\\
 &=\sum_k\lambda_k
   \left(\sum_i\ip{T_k}{H_i}H_i\right)
   \left(\sum_j\ip{T_k}{H_j}H_j\right)\\
 &=\sum_k\lambda_kT_k^2.
 \end{aligned}
\]
This expression is independent of the basis \(H_i\).
Since \(\tr(T_k^2)=\ip{T_k}{T_k}=1\),
\[
 \tr\Psi(Q)=\sum_k\lambda_k\tr(T_k^2)
 =\sum_k\lambda_k=\tr_VQ.
\]
If \(Q\succeq0\), then \(\lambda_k\ge0\) and each
\(T_k^2\succeq0\), so \(\Psi(Q)\succeq0\).
\end{proof}

We now encode the Jacobian second derivative in \eqref{eq:variations}
as an operator. For each unit \(x\), define \(Q_x:V\to V\)
by its matrix coefficients
\[
 (Q_x)_{ij}=n(n+2)(x^TH_ix)(x^TH_jx)-2n(H_ix)\cdot(H_jx).
\]
These coefficients are symmetric in \(i,j\), so \(Q_x\) is
self-adjoint. Expanding \(H\) in the basis \(H_i\) gives
\(\ip{Q_xH}{H}=J''_0(x)\).

\begin{lemma}\label{prop:contraction}
For every unit \(x\),
\begin{equation}\label{eq:contraction}
 \Psi(Q_x)=\frac{n-2}{2}(nxx^T-I),\qquad
 \tr_VQ_x=0.
\end{equation}
\end{lemma}
\begin{proof}
The matrix \(P=xx^T\) satisfies \(Pv=x(x\cdot v)\). Since \(x\)
is a unit vector, \(P\) is the orthogonal projection onto \(\R x\);
in particular, \(P^2=P\) and \(\tr P=1\).

Choose an orthonormal basis \((x,e_1,\dots,e_{n-1})\) of \(\R^n\).
In these coordinates, every \(H\in V\) has a unique representation
\[
 H=\begin{pmatrix}
 a&u^T\\
 u&C_0-\dfrac{a}{n-1}I_{n-1}
 \end{pmatrix},\qquad
 u\in\R^{n-1},\quad C_0^T=C_0,\quad \tr C_0=0.
\]
Here \(Hx=ax+\sum_i u_i e_i\). Thus \(Hx=0\) exactly when
\(a=0\) and \(u=0\): these are the matrices with \(x\) in their
kernel and an arbitrary symmetric trace-free block on \(x^\perp\).

Let \(V_0,V_u,V_a\) denote the subspaces obtained by varying only
\(C_0\), only \(u\), or only \(a\), with the other two set to zero.
Thus \(V_0=\{H\in V:Hx=0\}\), and
\[
 V_u=\left\{\begin{pmatrix}0&u^T\\u&0\end{pmatrix}
       :u\in\R^{n-1}\right\},\qquad
 V_a=\left\{a\begin{pmatrix}1&0\\0&-\dfrac1{n-1}I_{n-1}\end{pmatrix}
       :a\in\R\right\}.
\]
These three subspaces are mutually orthogonal and their sum is \(V\).
To determine how \(Q_x\) acts on them, let \(K\in V\) have
corresponding blocks \(D_0,v,b\). The Frobenius inner product
and the defining coefficients of \(Q_x\) give
\[
 \begin{aligned}
 \ip{H}{K}
 &=\ip{C_0}{D_0}+2u\cdot v+\frac{n}{n-1}ab,\\
 \ip{Q_xH}{K}
 &=n(n+2)(x^THx)(x^TKx)-2n(Hx)\cdot(Kx)\\
 &=n(n+2)ab-2n(ab+u\cdot v)\\
 &=n^2ab-2n\,u\cdot v.
 \end{aligned}
\]
The eigenvalues are obtained by comparing the coefficients in
\(\ip{Q_xH}{K}\) with those in \(\ip{H}{K}\). On \(V_u\),
the coefficient of \(u\cdot v\) changes from \(2\) to \(-2n\);
on \(V_a\), the coefficient of \(ab\) changes from \(n/(n-1)\)
to \(n^2\). Thus \(Q_x\) acts on \(V_0,V_u,V_a\) by
multiplication by, respectively,
\[
 0,\qquad \frac{-2n}{2}=-n,\qquad
 \frac{n^2}{n/(n-1)}=n(n-1).
\]
To apply the spectral formula for \(\Psi\), we use an orthonormal
basis \(T_1,\dots,T_{n-1}\) of \(V_u\) and a matrix \(H_0\)
spanning \(V_a\) with \(\|H_0\|_F=1\). They are given by
\[
 T_i=\frac{e_ix^T+xe_i^T}{\sqrt2},\qquad
 H_0=\frac{nP-I}{\sqrt{n(n-1)}}.
\]
Using \(e_i\perp x\), \(\sum_i e_ie_i^T=I-P\), and \(P^2=P\),
direct matrix multiplication gives
\[
 \begin{aligned}
 T_i^2&=\tfrac12(P+e_ie_i^T),\\
 \sum_iT_i^2&=\tfrac12((n-1)P+I-P)
             =\tfrac12(I+(n-2)P),\\
 n(n-1)H_0^2&=(nP-I)^2=I+n(n-2)P.
 \end{aligned}
\]
The spectral formula from Lemma~\ref{lem:positive-contraction} now yields
\[
 \begin{aligned}
 \Psi(Q_x)&=n(n-1)H_0^2-n\sum_iT_i^2\\
 &=I+n(n-2)P-\tfrac n2(I+(n-2)P)\\
 &=\frac{n-2}{2}(nP-I).
 \end{aligned}
\]
This proves the first identity. Since \(\tr P=1\), the resulting
matrix has trace zero. Lemma~\ref{lem:positive-contraction} then gives
\(\tr_VQ_x=\tr\Psi(Q_x)=0\).
\end{proof}

The pointwise identity becomes a statement about set measures by
integration. At \(s=0\), the second derivative of \(\mu(e^{sH}E)\)
is \(\ip{(\int_E Q_x\,d\mu(x))H}{H}\), while
\(\int_E(nxx^T-I)\,d\mu(x)=-G_E\). Linearity therefore gives the
identity we will use in the next arguments.

\begin{proposition}\label{lem:measure-contraction}
For every Borel set \(E\subset M\),
\begin{equation}\label{eq:measure-contraction}
 \Psi\left(\int_E Q_x\,d\mu(x)\right)=-\frac{n-2}{2}G_E.
\end{equation}
\end{proposition}

\subsection{Nonisotropy of optimizers}\label{sec:nonisotropy}
The contraction identity is useful for the curvature estimate only
if \(G_A\ne0\). When \(G_A=0\), every first measure derivative
vanishes and the eventual matrix inequality is trivial. We now
exclude this possibility for optimizers at second-order points,
using rotation averaging and the first nonzero Taylor term of the
measure variation.

\begin{proposition}\label{prop:nonisotropy}
At every second-order point \(p\in(0,1)\) of \(f\), every Borel
optimizer \(A\) satisfies \(G_A\ne0\).
\end{proposition}
\begin{proof}
Suppose that \(G_A=0\). Choose a full-measure Borel subset
\(B\subset M\setminus N(A)\), so that \((A,B)\) is an optimal
avoiding pair, and use the measure functions \(a,b\) from
\eqref{eq:ab}. Thus \(a(I)=p\) and \(b(I)=1-f(p)\).
Recall that \(f'(p)>0\).

The assumption \(G_A=0\) says that \(a\) has no first-order
change under symmetric deformations. However, by
Lemmas~\ref{lem:analytic-variations} and~\ref{lem:nonconstancy},
\(H\mapsto a(e^H)-p\) is analytic and is not identically zero
near \(0\in V\). It therefore has a first nonzero homogeneous
Taylor term of some degree \(m\ge2\):
\[
 a(e^H)-p=P_m(H)+O(\|H\|_F^{m+1}),\qquad P_m\not\equiv0.
\]
We will show that optimality forces this first nonzero term to
vanish as well.

To use optimality, we compare the changes in the two measures.
Feasibility gives \(b(T)\le1-f(a(T))\): the point
\((a(T),b(T))\) lies on or below the graph \(y=1-f(x)\).
At \((p,1-f(p))\), this graph has tangent line
\(y=1-f(p)-f'(p)(x-p)\). We subtract its height at
\(x=a(T)\) from \(b(T)\), defining
\[
 u(T)=b(T)-(1-f(p))+f'(p)(a(T)-p),\qquad v(H)=u(e^H).
\]
Since \(p\) is a second-order point,
\(f(p+\delta)=f(p)+f'(p)\delta+\tfrac12f''(p)\delta^2+o(\delta^2)\).
With \(\delta=a(T)-p\), this expansion and the degree-\(m\)
estimate for \(a(e^H)-p\) give, for some \(C,C'>0\),
\begin{equation}\label{eq:analytic-upper}
 \begin{aligned}
 u(T)&\le f(p)+f'(p)\delta-f(p+\delta)\le C\delta^2,\\
 v(H)&\le C'\|H\|_F^{2m}
 \end{aligned}
\end{equation}
near \(T=I\) and \(H=0\), respectively.

The upper bound forces the first nonzero homogeneous Taylor term
of \(v\), if its degree is less than \(2m\), to be nonpositive.
We use rotation averaging from Lemma~\ref{lem:rotation-average}
to show that this term must vanish. For \(H\in V\), we have
\(e^H\in SL(n,\R)\) and \(e^{RHR^{-1}}=Re^HR^{-1}\).
Expanding \(u\), applying the lemma with \(T=e^H\), and using
\(\int_{SO(n)}dR=1\), \(\mu(A)=p\), and \(\mu(B)=1-f(p)\),
we obtain
\[
 \begin{aligned}
 \int_{SO(n)}v(RHR^{-1})\,dR
 &=\int_{SO(n)}u(Re^HR^{-1})\,dR\\
 &=\mu(B)-(1-f(p))+f'(p)(\mu(A)-p)\\
 &=(1-f(p))-(1-f(p))+f'(p)(p-p)\\
 &=0.
 \end{aligned}
\]
If \(U_k\) were the first nonzero homogeneous Taylor term of
\(v\), with \(k<2m\), applying \eqref{eq:analytic-upper} to
\(tH\), dividing by \(t^k\), and letting \(t\downarrow0\)
would give \(U_k(H)\le0\). Applying the same procedure to the
rotation average gives
\[
 \int_{SO(n)}U_k(RHR^{-1})\,dR=0.
\]
The Taylor remainder is uniform on the compact rotation orbit of
\(H\), which justifies passing the limit under the integral.
The integrand is continuous and nonpositive. If it were negative
at any rotation, it would be negative on an open set of positive
Haar measure, contradicting its zero integral. Taking \(R=I\)
therefore gives \(U_k(H)=0\) for every \(H\), a contradiction.
We have proved
\begin{equation}\label{eq:analytic-vanishing}
 v(H)=O(\|H\|_F^{2m}).
\end{equation}

We now return to the measure of \(A\). Choose \(Z\in V\) with
\(\|Z\|_F=1\) and \(P_m(Z)\ne0\), and follow the deformation
\(A_t=e^{tZ}A\), paired with \(B_t=e^{-tZ}B\). Along this path,
\(a(e^{tZ})-p=t^mP_m(Z)+O(t^{m+1})\). To contradict this
expansion, we will use \eqref{eq:analytic-vanishing} to show that
its derivative is \(O(t^{2m-2})\); integration will then give
order \(2m-1>m\).

The contraction identity provides the link: it recovers the first
measure derivative from the first two derivatives of \(u\).
To apply it, we vary the pair \((A_t,B_t)\) in all matrix directions
\(H\in V\), while keeping its base point \(t\) fixed.
Let \(D(t)\in V\) and the self-adjoint
operator \(Q(t):V\to V\) be the gradient and Hessian at \(H=0\)
of \(H\mapsto u(e^He^{tZ})\), using the Frobenius inner product.
Thus, for every \(H\in V\),
\[
 \begin{aligned}
 \left.\frac{d}{ds}u(e^{sH}e^{tZ})\right|_{s=0}
 &=\ip{D(t)}{H},\\
 \left.\frac{d^2}{ds^2}u(e^{sH}e^{tZ})\right|_{s=0}
 &=\ip{Q(t)H}{H}.
 \end{aligned}
\]
With \(t\) and \(H\) fixed, varying \(s\) gives
\(a(e^{sH}e^{tZ})=\mu(e^{sH}A_t)\) and
\(b(e^{sH}e^{tZ})=\mu(e^{-sH}B_t)\), since
\((e^{sH}e^{tZ})^{-T}=e^{-sH}e^{-tZ}\).
The derivatives defining \(D(t)\) and \(Q(t)\) are therefore
the first and second derivatives at \(s=0\) of
\(\mu(e^{-sH}B_t)+f'(p)\mu(e^{sH}A_t)\); the remaining terms
in \(u\) are constant in \(s\). For the \(B_t\) term, replacing
\(H\) by \(-H\) reverses the first measure derivative and leaves
the second unchanged. The variation formulas give
\[
 \begin{aligned}
 D(t)&=f'(p) G_{A_t}-G_{B_t},\\
 Q(t)&=\int_{B_t}Q_x\,d\mu(x)+f'(p)\int_{A_t}Q_x\,d\mu(x).
 \end{aligned}
\]
The second identity follows from the directional second derivatives
for all \(H\), which determine a self-adjoint operator.
Proposition~\ref{lem:measure-contraction} gives
\(\Psi(Q(t))=-\tfrac{n-2}{2}(G_{B_t}+f'(p) G_{A_t})\).
Eliminating \(G_{B_t}\) with the formula for \(D(t)\) yields
\begin{equation}\label{eq:affine-contraction}
 (n-2)f'(p) G_{A_t}=\frac{n-2}{2}D(t)-\Psi(Q(t)).
\end{equation}
Since \((n-2)f'(p)>0\) and \(\Psi\) is a fixed linear map,
bounds on \(D(t)\) and \(Q(t)\) give a bound on \(G_{A_t}\),
and hence on \(\frac{d}{dt}a(e^{tZ})=\ip{G_{A_t}}{Z}\).
In particular, bounds of order \(O(|t|^{2m-2})\) suffice:
integrating the derivative would then give
\(a(e^{tZ})-p=O(t^{2m-1})\) for \(t>0\), contradicting its
nonzero degree-\(m\) Taylor term.

It remains to estimate \(D(t)\) and \(Q(t)\) from the vanishing
of \(v\). They differentiate \(u(e^He^{tZ})\), and the product
\(e^He^{tZ}\) need not be symmetric. Polar decomposition lets
us express it using the symmetric deformations already controlled
by \eqref{eq:analytic-vanishing}. Namely,
\(e^He^{tZ}=Re^{K_t(H)}\), where \(R\in SO(n)\) and
\(K_t(H)=\tfrac12\log(e^{tZ}e^{2H}e^{tZ})\in V\).
The matrix inside the logarithm is positive definite with
determinant one, so its logarithm is symmetric and trace-free.
By rotation invariance (Lemma~\ref{lem:rotation-invariance}),
\(u(e^He^{tZ})=v(K_t(H))\).

The logarithm power series shows that \(K_t(H)\) is analytic
near \((t,H)=(0,0)\); in particular, its first two derivatives
in \(H\) are uniformly bounded there. Also \(K_t(0)=tZ\).
Lemma~\ref{lem:analytic-variations} gives the analyticity of
\(v(H)=u(e^H)\), and \eqref{eq:analytic-vanishing} shows that
every Taylor term of \(v\) has degree at least \(2m\).
Differentiating its convergent Taylor series therefore gives
orders \(O(|t|^{2m-1})\) and \(O(|t|^{2m-2})\) for its first
and second derivatives at \(K_t(0)=tZ\), respectively.

Recall that \(D(t)\) and \(Q(t)\) were defined as the gradient
and Hessian of \(H\mapsto u(e^He^{tZ})\) at \(H=0\).
The identity \(u(e^He^{tZ})=v(K_t(H))\), obtained above from
Lemma~\ref{lem:rotation-invariance}, lets us apply the chain rule:
the first derivative uses the first derivative of \(v\), while
the second uses its first and second derivatives. Combining the
Taylor estimates just obtained with the uniform bounds on the
first two derivatives of \(K_t\) established above gives
\[
 D(t)=O(|t|^{2m-1}),\qquad Q(t)=O(|t|^{2m-2}).
\]
The estimate for \(D(t)=f'(p)G_{A_t}-G_{B_t}\) controls a
difference, whose two terms may cancel. To bound \(G_{A_t}\)
separately, we use \eqref{eq:affine-contraction}, which also
involves \(Q(t)\). Since \(\Psi\) is a fixed linear map and
\((n-2)f'(p)>0\), for sufficiently small \(0<|t|\le1\) we obtain
\[
 \begin{aligned}
 \|G_{A_t}\|_F
 &\le \frac{\|D(t)\|_F}{2f'(p)}
       +\frac{\|\Psi(Q(t))\|_F}{(n-2)f'(p)}\\
 &\le C_1|t|^{2m-1}+C_2|t|^{2m-2}\\
 &\le (C_1+C_2)|t|^{2m-2}.
 \end{aligned}
\]
Thus the \(Q(t)\) term gives the bound
\(G_{A_t}=O(|t|^{2m-2})\); the faster decay of \(D(t)\)
does not improve this estimate.

We have obtained the desired derivative bound. For small \(t>0\),
integrating along the chosen path and using \(\|Z\|_F=1\) gives
\[
 \begin{aligned}
 a(e^{tZ})-p&=\int_0^t\ip{G_{A_s}}{Z}\,ds,\\
 |a(e^{tZ})-p|&\le C\int_0^t s^{2m-2}\,ds
                =\frac{C}{2m-1}t^{2m-1}.
 \end{aligned}
\]
Dividing by \(t^m\) and letting \(t\downarrow0\) gives zero,
since \(2m-1>m\), whereas the Taylor expansion gives the nonzero
limit \(P_m(Z)\). This contradiction proves \(G_A\ne0\).
\end{proof}

\subsection{Differential comparison with the cap profile}\label{sec:differential}
We now compare the slopes of \(f\) and the cap profile \(g_n\).
Optimality first gives a matrix inequality. The nonzero moment matrix
from Proposition~\ref{prop:nonisotropy} then supplies a negative
eigenvalue, which we bound by comparing directional moments with
those of a cap.

\begin{proposition}\label{prop:differential}
At every second-order point \(p\in(0,1)\) of \(f\),
\[
 \frac{f''(p)}{f'(p)}\le\frac{g_n''(p)}{g_n'(p)}.
\]
In particular, this inequality holds almost everywhere in \((0,1)\).
\end{proposition}
\begin{proof}
Fix such a point \(p\) and choose an optimal avoiding pair \((A,B)\)
with \(\mu(A)=p\), by Proposition~\ref{prop:attainment}.
Write \(G=G_A\).
Proposition~\ref{prop:concavity} gives \(f'(p)>0\) and
\(f''(p)\le0\), while Proposition~\ref{prop:nonisotropy} gives
\(G\ne0\).

For every \(H\in V\), feasibility (Lemma~\ref{lem:feasibility})
gives \(\mu(e^{-sH}B)+f(\mu(e^{sH}A))\le1\), with equality at
\(s=0\). To expand this expression, collect the second measure
variations in the self-adjoint operator
\[
 Q=\int_B Q_x\,d\mu(x)+f'(p)\int_A Q_x\,d\mu(x).
\]
The variation formulas from Lemma~\ref{lem:analytic-variations}
and the Peano expansion \eqref{eq:Peano} give
\[
 \begin{aligned}
 &\mu(e^{-sH}B)+f(\mu(e^{sH}A))\\
 &\quad=1+s\ip{f'(p)G-G_B}{H}\\
 &\qquad+\frac{s^2}{2}
   \bigl(\ip{QH}{H}+f''(p)\ip{G}{H}^2\bigr)+o(s^2).
 \end{aligned}
\]
The remainder is \(o(s^2)\) because
\(\mu(e^{sH}A)-p=O(s)\), including when its first derivative
vanishes. Since the expression is at most one for both signs of
\(s\), the linear coefficient must vanish and the quadratic
coefficient must be nonpositive. Define the self-adjoint operator
\(L\) by \(LH=-f''(p)\ip{G}{H}G-QH\), so that the coefficient
of \(s^2/2\) is \(-\ip{LH}{H}\). We have proved
\begin{equation}\label{eq:stationarity}
 G_B=f'(p)G,\qquad L\succeq0.
\end{equation}

We apply \(\Psi\) to turn this positivity on \(V\) into an
inequality for \(n\times n\) matrices. Proposition~\ref{lem:measure-contraction}
and \eqref{eq:stationarity} give
\[
 \Psi(Q)=-\frac{n-2}{2}\bigl(G_B+f'(p)G\bigr)
        =-(n-2)f'(p)G.
\]
In the orthonormal basis used in \eqref{eq:Psi},
\(L_{ij}=-f''(p)\ip{G}{H_i}\ip{G}{H_j}-Q_{ij}\).
Substituting these coefficients and using
\(G=\sum_i\ip{G}{H_i}H_i\), we obtain
\[
 \begin{aligned}
 \Psi(L)
 &=-f''(p)\left(\sum_i\ip{G}{H_i}H_i\right)^2-\Psi(Q)\\
 &=-f''(p)G^2-\Psi(Q).
 \end{aligned}
\]
Lemma~\ref{lem:positive-contraction} gives \(\Psi(L)\succeq0\).
Consequently,
\begin{equation}\label{eq:matrix}
 -f''(p)G^2+(n-2)f'(p)G\succeq0.
\end{equation}

Since \(G\ne0\) is symmetric and its eigenvalues sum to zero,
it has a negative eigenvalue \(-\beta\), with \(\beta>0\).
Choose a unit eigenvector \(z\), so \(Gz=-\beta z\).
Evaluating \eqref{eq:matrix} on \(z\) gives
\(0\le-f''(p)\beta^2-(n-2)f'(p)\beta\).
Together with the definition of \(G_A\) in \eqref{eq:moments},
this yields
\begin{equation}\label{eq:rho}
 \frac{f''(p)}{f'(p)}\le-\frac{n-2}{\beta},\qquad
 \beta=n\int_A(x\cdot z)^2\,d\mu(x)-p.
\end{equation}
To obtain a bound depending only on \(p\), we therefore need an
upper bound for \(\beta\), or equivalently for the second moment
of \(A\) in direction \(z\).

Among sets of measure \(p\), the cap about \(z\) selects the
largest values of \((x\cdot z)^2\). Put \(r=A_n^{-1}(p)\),
\(u=\cos r\), and \(C=\{x:|x\cdot z|\ge u\}\).
The pointwise inequality
\(((x\cdot z)^2-u^2)(\one_A-\one_C)\le0\)
expresses this choice: the first factor is nonnegative on \(C\)
and negative outside it. On integration, the constant term cancels
because \(\mu(A)=\mu(C)=p\). Using the latitude density
\eqref{eq:latitude}, we obtain
\begin{equation}\label{eq:moment-bound}
 \int_A(x\cdot z)^2\,d\mu(x)
 \le\int_C(x\cdot z)^2\,d\mu(x)
 =\int_u^1t^2w_n(t)\,dt.
\end{equation}
Equality holds exactly when \(A=C\) almost everywhere: the
pointwise product is strictly negative on their symmetric difference
except on the null latitude \(|x\cdot z|=u\).
This is the fixed-measure form of the bathtub principle
\cite[Theorem 1.14]{LiebLoss}.

Since \(p=\int_u^1w_n(t)\,dt\), equations \eqref{eq:rho}
and \eqref{eq:moment-bound} give
\begin{equation}\label{eq:D}
 \begin{aligned}
 \beta\le D_n(p)
 &:=\int_u^1(nt^2-1)w_n(t)\,dt\\
 &=-c_n\bigl[t(1-t^2)^{(n-1)/2}\bigr]_u^1\\
 &=c_nu(1-u^2)^{(n-1)/2}>0.
 \end{aligned}
\end{equation}
Here \(0<u<1\), since \(0<p<1\). Substituting this upper bound
for \(\beta\) in \eqref{eq:rho} gives
\begin{equation}\label{eq:ODE}
 \frac{f''(p)}{f'(p)}
 \le-\frac{n-2}{\beta}
 \le-\frac{n-2}{D_n(p)}.
\end{equation}

It remains to identify the right-hand side with the cap profile.
By \eqref{eq:angular}, \(A_n'(r)=c_n\sin^{n-2}r\) and
\(B_n'(r)=c_n\cos^{n-2}r\). Differentiate
\(g_n(A_n(r))=B_n(r)\), then differentiate the logarithm of the
resulting slope. Since \(p=A_n(r)\), differentiation with respect
to \(p\) is differentiation with respect to \(r\) divided by
\(A_n'(r)\). Thus
\begin{equation}\label{eq:cap-ODE}
 \begin{aligned}
 g_n'(p)&=\frac{B_n'(r)}{A_n'(r)}=\cot^{n-2}r>0,\\
 \frac{g_n''(p)}{g_n'(p)}
 &=\frac{1}{A_n'(r)}\frac{d}{dr}\log\bigl(\cot^{n-2}r\bigr)\\
 &=-\frac{n-2}{c_n\cos r\,\sin^{n-1}r}
  =-\frac{n-2}{D_n(p)}.
 \end{aligned}
\end{equation}
Combining \eqref{eq:ODE} and \eqref{eq:cap-ODE} proves the claim.
\end{proof}

\subsection{Identification of the profile and rigidity}\label{sec:comparison}
\label{sec:representatives}
We complete the proof by integrating the differential comparison,
identifying the Borel optimizers, and passing to completed-measurable
sets. The last step will use a Borel subset of full measure for both
the inequality and its rigidity.

\begin{proof}[Proof of Theorem~\ref{thm:shadow}]
Proposition~\ref{prop:concavity} gives a positive, nonincreasing
right derivative \(v=f'_+\). On each compact subinterval of
\((0,1)\), it is bounded above and away from zero, so \(\log v\)
is also locally bounded and nonincreasing. Proposition~\ref{prop:differential}
gives
\[
 (\log v)'=\frac{f''}{v}\le\frac{g_n''}{g_n'}
 \quad\text{almost everywhere}.
\]
The function \(\log v\) need not be absolutely continuous: its
almost-everywhere derivative may miss jumps or singular decreases.
We therefore use the monotone integral inequality
\eqref{eq:monotone-integral}. For \(0<s<t<1\), it gives
\begin{equation}\label{eq:slope-comparison}
 \begin{aligned}
 \log v(t)-\log v(s)
 &\le\int_s^t(\log v)'(p)\,dp\\
 &\le\int_s^t\frac{g_n''(p)}{g_n'(p)}\,dp\\
 &=\log g_n'(t)-\log g_n'(s).
 \end{aligned}
\end{equation}
Rearranging and exponentiating shows that
\(v(t)/g_n'(t)\le v(s)/g_n'(s)\). Thus the ratio \(v/g_n'\)
is nonincreasing, including across any singular decreases of \(v\).

This ratio becomes a derivative if we use the cap neighborhood measure
as the independent variable. Set \(\Phi=f\circ g_n^{-1}\).
Concavity makes \(f\) locally Lipschitz in \((0,1)\), and
\(g_n^{-1}\) is smooth there, so \(\Phi\) is locally absolutely
continuous. The chain rule gives
\[
 \Phi'(s)=\frac{v(g_n^{-1}(s))}{g_n'(g_n^{-1}(s))}
 \quad\text{for almost every }s\in(0,1).
\]
Since \(g_n^{-1}\) is increasing, the right-hand side is a
nonincreasing representative of \(\Phi'\). The secant slopes of
\(\Phi\) are averages of this derivative, so \(\Phi\) is concave
in \((0,1)\). Continuity extends the concavity to \([0,1]\).
By \eqref{eq:elementary}, \(\Phi(0)=0\) and \(\Phi(1)=1\),
so its graph lies above the endpoint chord: \(\Phi(s)\ge s\).
Taking \(s=g_n(p)\) yields \(f(p)\ge g_n(p)\). Together with
the cap bound \(f\le g_n\) in \eqref{eq:elementary}, this proves
\begin{equation}\label{eq:profile}
 f=g_n\quad\text{on }[0,1].
\end{equation}
By the definition \eqref{eq:f}, this is the neighborhood inequality for
Borel sets. Caps attain equality by \eqref{eq:cap-shadow}.

For rigidity, let \(A\) be a Borel optimizer of measure
\(p\in(0,1)\). Choose a full-measure Borel subset \(B\) of
\(M\setminus N(A)\), so that \((A,B)\) is an optimal avoiding
pair. Since \(f=g_n\) and \(g_n\) is smooth in the interior,
\(p\) is now a second-order point. Proposition~\ref{prop:nonisotropy}
gives \(G_A\ne0\). As \(G_A\) is symmetric and trace-free,
choose a negative eigenvalue \(-\beta\), with \(\beta>0\),
and a unit eigenvector \(z\).

Apply the variational and moment bounds from \eqref{eq:ODE}
to this pair. The identity \(f=g_n\) and the cap equation
\eqref{eq:cap-ODE} make the two ends of the comparison equal:
\[
 \begin{aligned}
 -\frac{n-2}{D_n(p)}=\frac{f''(p)}{f'(p)}
 &\le-\frac{n-2}{\beta}\\
 &\le-\frac{n-2}{D_n(p)}.
 \end{aligned}
\]
Hence \(\beta=D_n(p)\). By the formulas for \(\beta\) and
\(D_n(p)\) in \eqref{eq:rho} and \eqref{eq:D}, this is equality
in the moment bound \eqref{eq:moment-bound} for direction \(z\).
Its equality case therefore forces \(A\) to agree almost everywhere
with the cap of measure \(p\) about \(z\).

Finally, let \(A\) be any completed-measurable set of measure
\(p\). By completion of the measure, choose a Borel subset
\(A_0\subset A\) with \(\mu(A\setminus A_0)=0\).
Its neighborhood \(N(A_0)\) is measurable, as noted in
\S\ref{sec:measure-profile}, and is contained in \(N(A)\).
The Borel bound therefore gives
\[
 \begin{aligned}
 \mu_*(N(A))
 &\ge\mu_*(N(A_0))\\
 &=\mu(N(A_0))\\
 &\ge g_n(\mu(A_0))=g_n(p).
 \end{aligned}
\]
If \(0<p<1\) and \(\mu_*(N(A))=g_n(p)\), all inequalities
in this chain are equalities. Thus \(A_0\) is a Borel optimizer,
so the rigidity just proved applies to \(A_0\). Since \(A\)
and \(A_0\) differ by a null set, it applies to \(A\) as well.
\end{proof}

\section{Corollaries and applications}\label{sec:corollaries}
The neighborhood inequality first gives a sharp bound for the measures of
two avoiding sets. Taking the sets equal yields the double-cap bound;
the Funk transform gives a corresponding bound for positivity sets.
Rigidity and compactness then yield stability, while the independence
bound determines the measurable fractional chromatic number and
the exponential growth rate of the measurable chromatic number.

\subsection{Avoiding pairs and the double-cap problem}
The profile determines the entire feasible region for two avoiding
sets, including pairs of unequal measures.

\begin{corollary}\label{cor:two-set}
If completed-measurable \(A,B\subset M\) avoid one another, and
\(p=\mu(A),q=\mu(B)\), then
\begin{equation}\label{eq:two-set}
 q\le1-g_n(p),\qquad
 A_n^{-1}(p)+A_n^{-1}(q)\le\pi/2.
\end{equation}
Every pair of measures satisfying this condition is realized by
avoiding sets. If \(0<p,q<1\) and equality holds, \(A\) and \(B\)
agree almost everywhere with caps having the same axis and
complementary radii. In addition,
\begin{equation}\label{eq:product}
 \mu(A)\mu(B)\le\delta_n^2,
\end{equation}
with equality only for two caps of radius \(\pi/4\) about the same
projective axis, up to null sets.
\end{corollary}
\begin{proof}
Avoidance means \(N(A)\subset M\setminus B\). Monotonicity of inner
measure and Theorem~\ref{thm:shadow} therefore give
\[
 g_n(p)\le\mu_*(N(A))\le1-q.
\]
Put \(r=A_n^{-1}(p)\) and \(s=A_n^{-1}(q)\).
Identity \eqref{eq:complement} converts this bound into
\(r+s\le\pi/2\). Conversely, open projective caps about a common
axis with these radii avoid one another: representatives with
positive inner product with the axis have mutual angle less than
\(r+s\). These caps realize every admissible pair of measures,
with the empty set used when a radius is zero.

If equality holds and \(0<p,q<1\), the displayed inequalities
force \(\mu_*(N(A))=g_n(p)\). The rigidity in
Theorem~\ref{thm:shadow} makes \(A\) a cap almost everywhere,
with some axis \(z\). To identify \(B\) without changing pointwise
avoidance, use density-one representatives. Avoidance gives zero
incidence, so Lemma~\ref{lem:compact} says that \(A^*\) and
\(B^*\) still avoid one another. Here \(A^*\) is exactly the open
cap of radius \(r\) about \(z\), whose neighborhood is the open belt
\(|y\cdot z|<\sin r\). Thus \(B^*\) lies in the complementary
closed cap of radius \(\pi/2-r\) about the same axis. Its measure
\(q=1-g_n(p)\) equals that cap's measure, so it fills the cap up
to a null set.

For the product bound, the case \(pq=0\) is immediate. Otherwise
\(r,s\in(0,\pi/2)\). We use strict concavity of \(\log A_n\)
to show that the product is largest when the radii are equal.
From \eqref{eq:angular},
\(A_n''(r)=(n-2)\cot r\,A_n'(r)\). Since \(\cot t>\cot r\)
for \(0<t<r\), we obtain
\[
 \begin{aligned}
 A_n'(r)
 &=c_n\int_0^r(n-2)\cot t\,\sin^{n-2}t\,dt\\
 &>(n-2)\cot r\,A_n(r),\\
 (\log A_n)''(r)
 &=\frac{A_n'(r)}{A_n(r)}
   \left((n-2)\cot r-\frac{A_n'(r)}{A_n(r)}\right)<0.
 \end{aligned}
\]
Midpoint concavity of \(\log A_n\), followed by monotonicity and
\(r+s\le\pi/2\), now gives
\[
 \begin{aligned}
 A_n(r)A_n(s)
 &\le A_n\!\left(\frac{r+s}{2}\right)^2\\
 &\le A_n(\pi/4)^2=\delta_n^2.
 \end{aligned}
\]
Equality requires \(r=s\) and \(r+s=\pi/2\), hence
\(r=s=\pi/4\). Equality in the two-set bound then identifies
the common axis.
\end{proof}

An independent set avoids itself. After passing from the sphere to
projective space, the two-set bound therefore gives the double-cap
bound and its equality case.

\begin{corollary}\label{cor:double}
For \(n\ge3\), \(\alpha_n=\delta_n\). Every measurable independent
\(E\subset S^{n-1}\) with \(\sigma(E)=\delta_n\) agrees almost
everywhere with two opposite caps of angular radius \(\pi/4\).
\end{corollary}
\begin{proof}
The antipodal saturation \(\widetilde E=E\cup(-E)\) remains
independent. Its projective image \(A\) is completed-measurable,
because the finite covering is locally a smooth diffeomorphism,
and \(p=\mu(A)=\sigma(\widetilde E)\ge\sigma(E)\).
Apply \eqref{eq:two-set} to \((A,A)\) to obtain
\(2A_n^{-1}(p)\le\pi/2\). Thus \(\sigma(E)\le\delta_n\).
Opposite open caps attain this measure. If equality holds, then
\(p=\delta_n\), and equality in the two-set bound makes \(A\)
a projective cap almost everywhere. Also
\(\sigma(\widetilde E\setminus E)=0\), proving the spherical claim.
\end{proof}
Corollary~\ref{cor:double} therefore resolves Witsenhausen's problem
for \(n\ge3\), proving the double-cap conjecture and characterizing
all maximizers up to null sets. Open caps are essential for the
pointwise construction: their boundaries can contain orthogonal pairs. Equality up to null
sets does not authorize arbitrary boundary additions. The argument
uses \(n\ge3\); the value \(\alpha_2=1/2\) follows from the
one-angle circle result of \cite{DeCortePikhurko}, where rigidity
takes a different form.

\begin{corollary}\label{cor:support}
Let \(h\in L^2(M)\) be nonnegative, and put
\(p=\mu(\{h>0\})\). Then
\begin{equation}\label{eq:support}
 \mu(\{Fh>0\})\ge g_n(p).
\end{equation}
The bound is sharp. If \(0<p<1\) and equality holds, then
\(\{h>0\}\) agrees almost everywhere with a cap.
\end{corollary}
\begin{proof}
Let \(E=\{h>0\}\) and \(Z=\{Fh=0\}\), understood
modulo null sets. Self-adjointness gives
\(0=\ip{\one_Z}{Fh}=\int_M h\,F\one_Z\,d\mu\).
Nonnegativity implies that
\(F\one_Z=0\) almost everywhere on \(E\), and
thus \(E\) and \(Z\) have zero incidence. Lemma~\ref{lem:compact}
makes \(E^*,Z^*\) an avoiding pair, so Corollary~\ref{cor:two-set}
gives \(1-\mu(Z)\ge g_n(p)\). Its equality case shows that,
when \(0<p<1\), equality forces \(E\) to be a cap almost everywhere.
Cap indicators attain the bound: their Funk averages are positive
on the interior of the corresponding belt and vanish outside it,
up to its null boundary.
\end{proof}

\subsection{Qualitative stability}
Rigidity and compactness imply that a set whose neighborhood measure is
almost minimal must be close in measure to some cap.

\begin{corollary}\label{cor:stability}
Fix \(p\in(0,1)\). For every \(\varepsilon>0\) there is
\(\eta=\eta(n,p,\varepsilon)>0\) such that
\[
 \mu(A)=p,\quad \mu_*(N(A))\le g_n(p)+\eta
 \quad\Longrightarrow\quad
 \inf_{z\in S^{n-1}}
 \mu\bigl(A\triangle\cC_{A_n^{-1}(p)}(z)\bigr)<\varepsilon
\]
for every completed-measurable \(A\subset M\).
\end{corollary}
\begin{proof}
It suffices to consider Borel \(A\): a full-measure Borel subset
has a neighborhood of no larger inner measure and the same distance in
symmetric-difference measure from every cap.
If the assertion fails, choose Borel sets \(A_j\) of measure \(p\)
whose neighborhoods satisfy \(\mu(N(A_j))\to g_n(p)\), while
\(\mu(A_j\triangle C)\ge\varepsilon\) for every cap \(C\)
of measure \(p\).

Pass to a subsequence with \(\one_{A_j}\rightharpoonup u\) weakly
in \(L^2(M)\). Then \(0\le u\le1\) and \(\int_M u\,d\mu=p\),
so \(D=\{u>0\}\) has measure at least \(p\).
Compactness of \(F\) gives \(F\one_{A_j}\to Fu\) strongly in
\(L^2\) and, after a further subsequence, almost everywhere.
Where \(Fu>0\), these averages are eventually positive.
Since \(F\one_{A_j}\) vanishes outside \(N(A_j)\), Fatou's
lemma and Corollary~\ref{cor:support} give
\[
 g_n(p)\le g_n(\mu(D))\le\mu(\{Fu>0\})
 \le\liminf_j\mu(N(A_j))=g_n(p).
\]
Strict increase of \(g_n\) forces \(\mu(D)=p\).
As \(0\le u\le1\) and \(\int_M u\,d\mu=\mu(D)\), we have
\(u=\one_D\) almost everywhere. Equality in
Corollary~\ref{cor:support} makes \(D\) a cap up to a null set.

Finally, weak convergence to an indicator of the same measure is
strong. Indeed,
\[
 \begin{aligned}
 \mu(A_j\triangle D)
 &=\|\one_{A_j}-\one_D\|_2^2\\
 &=2p-2\ip{\one_{A_j}}{\one_D}\longrightarrow0.
 \end{aligned}
\]
This contradicts the separation of \(A_j\) from every cap of
measure \(p\).
\end{proof}
An explicit relation between \(\eta\) and \(\varepsilon\) would be
a quantitative strengthening; the compactness argument alone
does not provide one.

\subsection{Measurable colorings of orthogonality graphs}\label{subsec:colorings}
The spherical orthogonality graphs form a natural compact counterpart
to Euclidean unit-distance graphs. For this spherical family, we
determine the measurable fractional chromatic number in every dimension
and the exponential growth rate of the measurable chromatic number.
The corresponding Euclidean questions remain open: the exact measurable
fractional value is unknown even in the plane \cite[\S4]{CranstonRabern}, and the
precise exponential growth rate of the measurable chromatic number has
yet to be determined \cite[\S1]{Davies}.
Orthogonality corresponds to unit distance on a sphere of radius
\(1/\sqrt2\). More general spherical distance graphs were studied
by Lov\'asz~\cite{LovaszSphere}, who obtained topological lower
bounds for their chromatic numbers.

Recall that \(\mathcal G_n\) has vertex set \(S^{n-1}\), with
orthogonal vertices adjacent. For a graph \(G\) on a measured space,
write \(\chi_{\mathrm m}(G)\) for the least number of measurable
independent sets in a partition of its vertices.
A measurable fractional coloring allows a finite family of measurable
independent sets \(I_1,\ldots,I_k\) to overlap, assigning weights
\(w_j\ge0\) such that
\begin{equation}\label{eq:fractional-definition}
 \sum_{j=1}^k w_j\one_{I_j}(x)\ge1
 \qquad\text{for every vertex }x.
\end{equation}
The infimum of the total weight \(\sum_jw_j\) is denoted by
\(\chi_{f,\mathrm m}(G)\); see \cite{BachocEtAl}.

The fractional value follows from a general reciprocal identity.
For finite vertex-transitive graphs it is classical
\cite[Proposition 3.1.1]{ScheinermanUllman}; the general measurable
lower bound is \cite[Proposition 2.1]{BachocEtAl}.
We include the compact version for completeness. For its upper bound,
we average translates of an independent set whose measure is close
to the supremum, then use compactness to replace the average by a
finite weighted cover of every vertex.

\begin{proposition}\label{prop:homogeneous}
Let \(X\) be a compact metric space with a continuous transitive
action of a compact group \(\Gamma\), and let \(\lambda\) be its
invariant Borel probability measure. Suppose \(G\) is a graph on
\(X\) preserved by this action, and its ordered edge relation is
closed in \(X\times X\) and disjoint from the diagonal. Here the
ordered edge relation is \(\{(x,y):x\text{ is adjacent to }y\}\),
and the diagonal is \(\{(x,x):x\in X\}\). If
\[
 \alpha_\lambda(G)=
 \sup\{\lambda(I):I\subset X\text{ measurable and independent}\},
\]
then
\begin{equation}\label{eq:homogeneous}
 \chi_{f,\mathrm m}(G)=\frac1{\alpha_\lambda(G)}.
\end{equation}
\end{proposition}
\begin{proof}
Write \(\alpha=\alpha_\lambda(G)\). The assumptions imply
\(\alpha>0\): the complement of the edge relation is open and
contains \((x,x)\), so each point has an independent open
neighborhood. Every nonempty open set has positive invariant
measure, since finitely many of its translates cover \(X\).
Integrating any fractional cover gives
\[
 1\le\sum_jw_j\lambda(I_j)\le\alpha\sum_jw_j.
\]
This proves the lower bound, even when coverage holds only
almost everywhere.

For the reverse bound, fix \(0<\varepsilon<\alpha\).
Inner regularity gives a compact independent set \(K\) with
\(\lambda(K)>\alpha-\varepsilon\). Since the edge relation is
closed and disjoint from the compact set \(K\times K\), there
is an open independent neighborhood \(U\) of \(K\).
This small enlargement allows nearby translates of \(K\) to be
covered by a single translate of \(U\).

More precisely, compactness of \(K\) and continuity of the action
give an open neighborhood \(W\) of the identity in \(\Gamma\)
with \(WK\subset U\). Choose a finite cover
\(\Gamma=\bigcup_{j=1}^k\gamma_jW\), and refine it to a Borel
partition \(V_j\subset\gamma_jW\). If \(\gamma\in V_j\), write
\(\gamma=\gamma_jw\) with \(w\in W\); then
\(\gamma K=\gamma_j(wK)\subset\gamma_jU\).

For normalized Haar measure \(m_\Gamma\) and fixed \(x\),
transitivity implies that \(\gamma^{-1}x\) has distribution
\(\lambda\). Hence the average of \(\one_{\gamma K}(x)\) is
\(\lambda(K)\). Splitting that average over the partition and
using the inclusion above gives, at every \(x\in X\),
\[
 \begin{aligned}
 \lambda(K)
 &=\sum_{j=1}^k\int_{V_j}\one_{\gamma K}(x)\,dm_\Gamma(\gamma)\\
 &\le\sum_{j=1}^k m_\Gamma(V_j)\one_{\gamma_jU}(x).
 \end{aligned}
\]
The sets \(\gamma_jU\) are independent. Dividing by
\(\lambda(K)>0\) therefore gives a finite fractional cover with
weights \(w_j=m_\Gamma(V_j)/\lambda(K)\) and total cost
\(\sum_jw_j=\frac1{\lambda(K)}<\frac1{\alpha-\varepsilon}\).
Letting \(\varepsilon\downarrow0\) proves the reverse bound.
\end{proof}

For the spherical orthogonality graph, the double-cap theorem gives
\(\alpha_\sigma(\mathcal G_n)=\delta_n\), so
Proposition~\ref{prop:homogeneous} determines its measurable fractional
chromatic number. Since every measurable coloring is also a fractional
coloring, this gives a lower bound for the measurable chromatic number.
Combined with classical spherical covering bounds, this lower bound
also determines its exponential growth rate.

\begin{corollary}\label{cor:colorings}
For \(n\ge3\),
\begin{align}
 \chi_{f,\mathrm m}(\mathcal G_n)
   &=\delta_n^{-1},\label{eq:fractional-value}\\
 \chi_{\mathrm m}(\mathcal G_n)
   &\ge\left\lceil\delta_n^{-1}\right\rceil,
 \qquad
 \chi_{\mathrm m}(\mathcal G_n)=(\sqrt2+o(1))^n.
 \label{eq:chromatic-bound}
\end{align}
Both coloring parameters have the same values on the projective
orthogonality graph \(\mathcal P_n\).
\end{corollary}
\begin{proof}
The action of \(SO(n)\) on \(S^{n-1}\) is transitive, and
orthogonality is a closed relation disjoint from the diagonal.
Proposition~\ref{prop:homogeneous} therefore applies. By
Corollary~\ref{cor:double} and the opposite-cap construction,
\(\alpha_\sigma(\mathcal G_n)=\delta_n\), which proves
\eqref{eq:fractional-value}. A measurable coloring is a fractional
cover with all weights one, so it requires at least
\(\lceil\delta_n^{-1}\rceil\) colors.

Since \(c_n^{1/n}\to1\), the integral in \eqref{eq:delta} gives
\(\delta_n^{1/n}\to1/\sqrt2\): for a lower estimate integrate over
\([\pi/4-\varepsilon,\pi/4]\), then let \(\varepsilon\downarrow0\).
This proves the lower exponential bound in
\eqref{eq:chromatic-bound}. For the upper bound, fix \(0<r<\pi/4\).
Rogers' spherical covering bound \cite{RogersCovering}, also discussed
in \cite[\S1]{Prosanov}, provides a cover of \(S^{n-1}\) by
\((1/\sin r+o(1))^n\) caps of angular radius \(r\).
Each cap is independent; assigning each point to the first cap
containing it gives a Borel coloring. Hence
\(\limsup_n\chi_{\mathrm m}(\mathcal G_n)^{1/n}\le1/\sin r\).
Let \(r\uparrow\pi/4\).

Finally, replacing each independent set by its antipodal saturation
preserves measurability and independence. The projected sets give a
fractional cover of \(\mathcal P_n\) with the same weights; for an
integer coloring, make the projected cover disjoint by assigning each
point to the first set containing it. Conversely, projective covers
lift to the sphere. Both parameters therefore agree on the two graphs.
\end{proof}

The earlier independence bounds of Frankl--Wilson and Raigorodskii
\cite{FranklWilson,Raigorodskii} gave exponential lower bounds for
\(\chi_{\mathrm m}(\mathcal G_n)\) with bases \(1.13\) and \(1.225\),
respectively. Corollary~\ref{cor:colorings} matches the upper rate
\(\sqrt2\) supplied by spherical coverings. In fixed dimensions,
reciprocating and rounding the independence bounds of Bekker et al.\
\cite[Table~1]{BekkerEtAl} gives the following comparison:
\begin{center}
\begin{tabular}{@{}lrrrr@{}}
\toprule
Ambient dimension \(n\) & 5 & 6 & 7 & 8\\
\midrule
From Bekker et al. & 8 & 11 & 14 & 17\\
From Corollary~\ref{cor:colorings} & 9 & 14 & 21 & 31\\
\bottomrule
\end{tabular}
\end{center}
These are lower bounds for \(\chi_{\mathrm m}\). Sharpness of the
measure bound for a single class does not imply equality after rounding,
since the classes must also fit together in a partition.

On \(S^2\), the two measurable parameters are
\[
 \chi_{f,\mathrm m}(\mathcal G_3)=2+\sqrt2,
 \qquad \chi_{\mathrm m}(\mathcal G_3)=4.
\]
The integer value was proved by Simmons~\cite{Simmons}, whose
paper includes a Borel four-coloring due to J. Davis; see also
Godsil and Zaks~\cite{GodsilZaks}. For the fractional value, the bound
\(\alpha_3\le0.2953\) reported by Czifra et al.\ \cite{CzifraEtAl},
together with the double-cap construction and the reciprocal identity,
gave \(1/0.2953\le\chi_{f,\mathrm m}(\mathcal G_3)\le2+\sqrt2\).
Corollary~\ref{cor:colorings} identifies the upper endpoint as the
exact value.

Removing the measurability requirement from the fractional definition
gives \(\chi_f(G)\). Its finitary counterpart is
\(\chi_f^{\mathrm{fin}}(G)=\sup_{H\subset G,\ H\text{ finite}}\chi_f(H)\).
Thus
\begin{equation}\label{eq:fractional-comparison}
 \chi_f^{\mathrm{fin}}(\mathcal G_n)\le\chi_f(\mathcal G_n)
 \le\chi_{f,\mathrm m}(\mathcal G_n)=\delta_n^{-1}.
\end{equation}
The theorem determines the last quantity; equality with either of the
first two requires a further argument. For general infinite graphs,
the first inequality can be strict \cite{Leader}.
Likewise, if \(\chi(G)\) denotes the chromatic number with arbitrary
color classes, then \(\chi(\mathcal G_n)\le\chi_{\mathrm m}(\mathcal G_n)\).
For arbitrary colorings, exponential lower bounds follow from finite
configurations~\cite{RaigorodskiiSphere,Kostina}, but the measure argument
above does not determine the precise exponential growth rate of
\(\chi(\mathcal G_n)\).

\end{document}